\pdfoutput=1

\documentclass[11pt,letterpaper]{amsart}
\usepackage{LBT-macros,LBT-style}
\usepackage{multirow}

\title{A new approach to light bulb tricks: Disks in 4-manifolds}

\author{Danica Kosanovi\'c}
\address{LAGA, Universit\'e Sorbonne Paris Nord (Paris 13)}
\curraddr{Department of Mathematics, ETH Z\"urich}
\email{danica.kosanovic@math.ethz.ch}

\author{Peter Teichner}
\address{Max-Planck-Institut f\"ur Mathematik, Bonn}
\email{teichner@mac.com}

\begin{document}

\begin{abstract}
    For a 4-manifold $M$ and a knot $k\colon\S^1\hra\partial M$ with dual sphere $G\colon\S^2\hra\partial M$, we compute the set $\Dk$ of smooth isotopy classes of neat embeddings $\D^2\hra M$ with boundary $k$, using an invariant going back to Dax.
    Moreover, we construct a group structure on $\Dk$ and show that it is usually neither abelian nor finitely generated.
    We recover all previous results for isotopy classes of spheres with framed duals and relate the group $\Dk$ to the mapping class group of $M$.
\end{abstract}

\maketitle

\section{Introduction and main results}\label{sec:intro}

The failure of the Whitney trick makes any question about embedded surfaces in 4-manifolds notoriously difficult to answer. Under certain assumptions, including the existence of algebraic dual spheres, Freedman's disk/sphere embedding theorems construct topological embeddings, and in some cases one can also determine the topological isotopy classes. It therefore came as a huge surprise to the 4-manifold community, when Dave Gabai proved his light bulb theorem (LBT) for embedded spheres $\S^2\hra N^4$ with common embedded dual sphere in the \emph{smooth} category~\cite{Gabai-spheres}. This was inspired by the well-known 3-dimensional LBT that says that the homotopy and isotopy classification of knots $\S^1\hra N^3$ with common dual sphere agree for any 3-manifold $N$. Our main result is the following LBT, generalizing the simply-connected case in \cite[Thm.0.6(i)]{Gabai-disks}. 

Two submanifolds are \emph{dual} if they intersect transversely in a single point. A map $K\colon X\to Y$ between smooth manifolds is \emph{neat} if it is smooth, intersects $\partial Y$ transversely and $K^{-1}(\partial Y)=\partial X$. Throughout this paper $M$ will denote a smooth, oriented 4-manifold with nonempty boundary, and embeddings and isotopies are smooth.
\begin{theorem}\label{thm-intro:Dax-classifies}
    For a knot $k\colon\S^1\hra\partial M$ that has a dual sphere in $\partial M$, 
    two neat embeddings $\D^2\hra M$ with boundary $k$ are isotopic (rel.\ $k$) if and only if they are homotopic (rel.\ $k$) and their relative $\Dax$ invariant vanishes.
\end{theorem}
In Section~\ref{sec:spheres} we show how this LBT for disks implies all previous results for spheres \cite{Gabai-spheres,ST-LBT}, but for its proof we use very different techniques, e.g.\ the work of Dax in the 1970's on homotopy groups of embedding spaces~\cite{Dax}. It turns out that the classification for disks is more subtle in the sense that the $\Dax$ invariant for disks realizes all values in a much larger group than the Dax invariant for spheres, see Example~\ref{ex:boundary-conn-sum}. It is an open problem whether Theorem~\ref{thm-intro:Dax-classifies} continues to hold if a common dual sphere only exists in the interior of $M$ but see \cite{Schwartz-disks} for a special case. 

Even though more general, we believe that our proof is the simplest and most conceptual approach to the LBT.
We realized that space level methods of Cerf and Dax, extended to all dimensions in \cite{KT-highd} and recalled in Section~\ref{sec:space-level} below, can be combined to not only construct the Dax invariant $\Dax$ but show that it \emph{completely classifies isotopy}. We can also identify the image of $\Dax$ and describe the set of isotopy classes geometrically, see Theorems~\ref{thm-intro:Dax-values} (which directly implies \ref{thm-intro:Dax-classifies}) and~\ref{prop-intro:4-term}. Since this classification follows by delooping our space of embedded disks, we also get a group structure on the set of isotopy classes, see Theorem~\ref{thm-intro:groups}. In the rest of the introduction, apart from stating these results (with proofs in Section~\ref{sec:main-proofs}), we also explain the relative Dax invariant for general neat disks and give some examples.

\subsection{The relative Dax invariant for neat disks}\label{sec-intro:Dax}

Assume $M$ is an arbitrary smooth, oriented 4-manifold with a basepoint in $\partial M$ and 
fundamental group $\pi\coloneqq\pi_1M$. Let $\RGR$ be the free abelian group on the set of its nontrivial elements, with the usual involution induced by $\sigma(g)=\ol{g}\coloneqq g^{-1}$. Following work of Jean-Pierre Dax~\cite{Dax} we define a homomorphism
\[
    \md\colon \pi_3(M)\to \RGR^\sigma\coloneqq\{r\in \RGR: r=\ol{r}\}
\]
of abelian groups, recently also made explicit by Gabai~\cite{Gabai-disks}. The element $\md(a)$ counts self-intersections of a special kind of generic immersion $\D^3\imra\D^2 \times M^4$ obtained from $a\colon\S^3\to M$, with fundamental group elements defined by a clever choice of sheets at each double point, avoiding the usual indeterminacy $g=-\ol{g}$ for self-intersection invariants in dimension~6.

The homomorphism $\md$ is an interesting invariant of the 4-manifold $M$, only depending on the component of $\partial M$ that contains the basepoint. In Theorem~\ref{thm:Hopf} we show that $\md$ combines Wall's self-intersection invariants $\mu_2$ and $\mu_3$ for maps of 2- respectively 3-spheres to $M$, enabling computations. In particular, in Proposition~\ref{prop:Whitehead} we compute the Whitehead product of $a_i\in\pi_2M$ in terms of the reduced intersection form $\lambdabar$ of $M$ (see Proposition~\ref{prop:commutators}):
\begin{equation}\label{eq:Whitehead}
     \md_u\big([a_1,a_2]_{\mathsf{Wh}}\big)=\lambdabar(a_1,a_2)+\lambdabar(a_2,a_1).
\end{equation}
Moreover, in \cite{K-Dax} the surprising behavior of $\md$ under the $\pi$-action on $\pi_3M$ was exhibited, with an important special case as follows.
\begin{example}\label{ex:interior-conn-sum}
    In the connected sum $M=(\D^2 \times \S^2) \# M'$ the separating 3-sphere $C$ has $\md(C)=0$ as it is embedded. 
    However, the class $g\cdot [C]\in \pi_3M$, obtained by the usual $\pi$-action on $\pi_3M$, is represented by adding a tube along $g$ to $C$, and is no longer embedded.
    We will show that $\md(g\cdot [C])= g + \ol{g}$ in Lemma~\ref{lem:g+g-inv}. This will be used in Section~\ref{sec:spheres-proofs} to prove that for $M$ a connected sum as above, the Dax invariant reduces to the Freedman--Quinn invariant of~\cite{ST-LBT}, suitably adapted to disks in Definition~\ref{def:FQ}.
\end{example}
Now fix a knot $k\colon\S^1\hra\partial M$, not necessarily having a dual. Then the count of Dax also works for a homotopy $H$ (rel.\ boundary) between neat embeddings $K_0, K_1\colon\D^2\hra M$ with $\partial K_0=\partial K_1=k$, leading to the sum of signed group elements $\Dax(H)\in \RGR^\sigma$. This behaves additively under gluing homotopies, and the element $\md(a)$ above results from the special case where $H$ is the self-homotopy obtained from $K_0 \times I$ by an interior connected sum with $a\colon\S^3\to M$. As a consequence, the class $\Dax(K_0,K_1)\coloneqq [\Dax(H)]\in \RGR^\sigma/\md(\pi_3M)$ is independent of $H$; cf.\ {\cite[Cor.0.5]{Gabai-disks}}.
\begin{lemma}\label{lem:Dax-invt}
    The relative Dax invariant $\Dax(K_0,K_1)\in \RGR^\sigma/\md(\pi_3M)$ is defined for homotopic neat embeddings $K_0, K_1\colon\D^2\hra M$ and is an obstruction for the existence of an isotopy  between them (all rel.\ boundary). 
\end{lemma}
A similar definition is given for so-called \emph{half-disks} in Lemma~\ref{lem:HDax}, where the value of the $\Dax$ invariant is in general not fixed by the involution. 

In \cite[Fig.2, Thm.4.9]{Gabai-disks} the $\Dax$ invariant relative to $K_0$ was realized, using the simplest type of concordance on $K_0$ with a unique minimum and index 1 critical point: For $g\in\pi\sm 1$, start with an unknotted sphere $S_0\subseteq\R^4$ next to a basepoint of $K_0$ and add a tube along an embedded arc $A$ that connects $K_0$ to $S_0$ to obtain a disk $K_g$. This arc $A$ first runs from $K_0$ through $S_0$ (i.e.\ it intersects once the 3-ball bounding $S_0$), then follows $g$ and ends on $S_0$. The resulting \emph{self-referential} disk satisfies $\Dax(K_0,K_g) = [g + \bar g]$. 

Note that dual spheres are not required for the definition of $\Dax$ nor the self-referential disks.
But the vanishing of $\Dax(K_0,K_1)$ is obviously not a sufficient condition for $K_i$ to be isotopic; e.g.\ for $M=\D^4$ the fundamental group of the complements is a simple additional invariant. 


\subsection{Classification of neat disks via a geometric group action} \label{sec-intro:4-term-ex-seq}
We provide new realization and classification results for the $\Dax$ invariant in the presence of a dual sphere in the boundary.
\begin{theorem}\label{thm-intro:Dax-values} 
      Assume that $k\colon\S^1\hra\partial M$ is null homotopic in $M$ and has a dual $G\colon\S^2\hra\partial M$ (i.e.\ $k\pitchfork G=pt$). Then the set $\Dk$ of isotopy classes of neat embeddings $\D^2\hra M$ with boundary $k$ is nonempty and
\begin{enumerate}
    \item  the abelian group $\RGR$ acts on $\Dk$ via finger moves and Norman tricks, written $(r,K)\mapsto K+\fm(r)^G$, see Figure~\ref{fig:action};
    \item  the action of $\RGR^\sigma$ preserves the homotopy class of disks (rel.\ boundary), is transitive on disks in the same homotopy class and has stabilizer subgroups $\md(\pi_3M) \leq \RGR^\sigma$, independently of $K$;
    \item  the relative $\Dax$ invariant inverts this action in the sense that for $r\in \RGR^\sigma$ and $K\in\Dk$ we have $\Dax(K,K+\fm(r)^G)=[r]$.
\end{enumerate}
\end{theorem}
This result implies Theorem~\ref{thm-intro:Dax-classifies}, and in addition gives all possible values of $\Dax$ and a geometric construction of each value.
In particular, if $\pi_1M=1$ we obtain a bijection $\Dk \cong [\D^2,M;k]$, with the set of homotopy classes of maps $\D^2\to M$ with boundary $k$, reproving \cite[Thm.0.6(i)]{Gabai-disks}. 

Our geometric action $(r,K)\mapsto K+\fm(r)^G$ 
of $\RGR$ on $\Dk$ is given as follows, see Figure~\ref{fig:action}. First perform finger moves to $K$ along the group elements that make up $r\in \RGR$ (each finger move guided by $g\in \pi$ introduces a pair of double points with double point loop $g$), and then revert this immersion back into an embedding by adding a tube for each double point (along distinct choices of sheets for the given double point pair) into a parallel copy of $\pm G$. This tubing into the dual is often called the \emph{Norman trick}, and was used in~\cite{ST-LBT} to obtain embedded {Norman spheres}.
\begin{figure}[!htbp]
    \centering
    \includegraphics[width=0.88\linewidth]{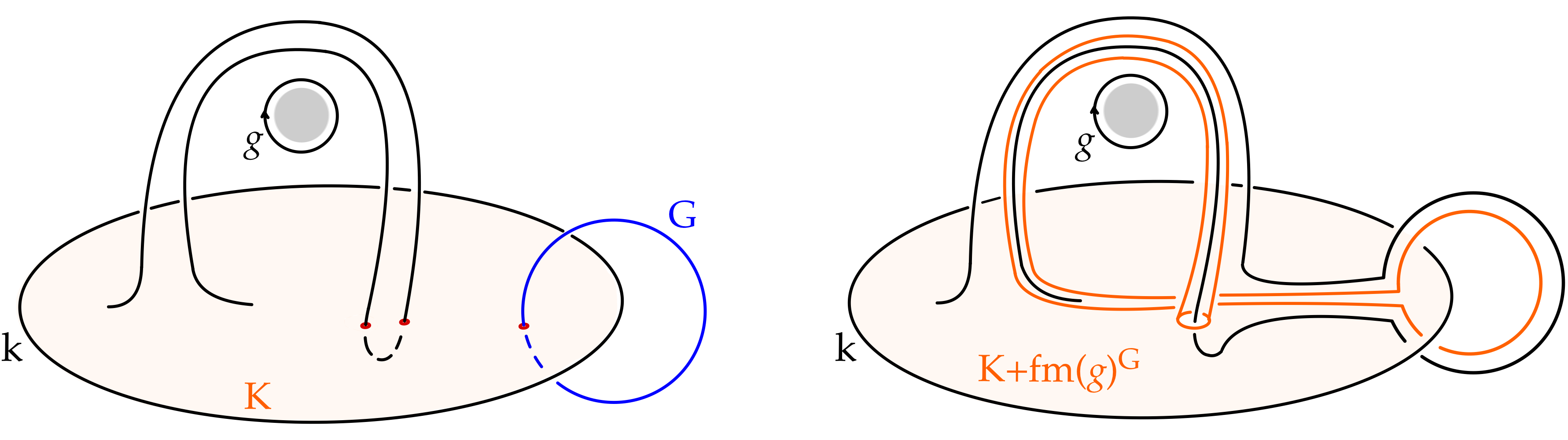}
    \caption{Constructing isotopy classes of disks with fixed boundary $k$ by finger moves along $g$ and Norman tricks on distinct sheets.}
    \label{fig:action}
\end{figure}

On homotopy classes of disks with boundary $k$, we will see in Lemma~\ref{lem:homotopy-class} 
\[
    [K+\fm(r)^G]=[K \# (\ol{r}-r)\cdot G]\in [\D^2,M;k], 
\]
where $\#$ denotes the interior connected sum, cf.\ Figure~\ref{fig:K-tw}. The right hand side equals $[K]$ precisely for $r\in \RGR^\sigma$ by Lemma~\ref{lem:lambda-splits} which explains the appearance of this subgroup in Theorem~\ref{thm-intro:Dax-values}.

The next result is a convenient way to express all information about our group action, using the following notation.
\begin{defn}\label{def:action}
    For a group $G$ and a set $S$ a group action $a\colon G \times S \to S$ with orbit set $S/G$ will be denoted by 
    $\begin{tikzcd}[cramped,column sep=17pt] 
        G\arrow[squiggly]{r}{a} & S\arrow[two heads]{r} & S/G 
    \end{tikzcd}$.
 
\end{defn}
\begin{prop}\label{prop-intro:4-term}
    Assume $k\colon\S^1\hra\partial M$ is null homotopic in $M$ and has a dual $G\colon\S^2\to\partial M$. Then our $\RGR^\sigma$-action $(+\fm^G)(r,K)\coloneqq K + \fm(r)^G$ on $\Dk$ from Theorem~\ref{thm-intro:Dax-values} fits into a 4-term exact sequence
    \[\begin{tikzcd}
        \faktor{\RGR^\sigma}{\md(\pi_3M)} \arrow[tail,squiggly]{r}{+\fm^G} 
        &\Dk \arrow{r}{j} 
        & \mbox{}[\D^2,M;k] \arrow[two heads]{r}{\mu_2} 
        & \faktor{\RGR}{\langle\ol{r}{-}r\rangle}
    \end{tikzcd}\]
    where $j$ takes the underlying homotopy class and $\mu_2$ is Wall's reduced self-intersection invariant. The exactness means from left to right: 
\begin{itemize}
\item 
    Every disk in $\Dk$ has stabilizer $\md(\pi_3M) \leq   \RGR^\sigma$,
\item
  $j$ factors through an injection of $\Dk/\RGR^\sigma$,
\item
    $\im(j)=\mu_2^{-1}(0)$,
\item
    $\mu_2$ is surjective.
\end{itemize}
\end{prop}

The invariant $\mu_2$ counts double points (with signs and fundamental group elements, away from the trivial element) of a generic immersion $J\colon\D^2\imra M$ representing the given homotopy class $[J]\in[\D^2,M;k]$. The element $\mu_2[J]$ is an obstruction for representing $[J]$ by an embedding and exactness means that it is the only obstruction in our setting, thanks to the Norman trick along $G$, see the proof of Proposition~\ref{prop-intro:4-term} in Section~\ref{sec:main-proofs}. 

For any disk $K\in\Dk$, note that gluing $-K$ to a map $\D^2\to M$ along the boundary gives a natural bijection 
    \begin{equation} \label{eq:cup-U}
        -K\cup\bull\colon [\D^2,M;k]\overset{\cong}{\ra}\pi_2M.
    \end{equation}
The following class of examples exhibits the richness of $\Dax$. For a framed, null homotopic knot $c\colon\S^1 \hra \S^1 \times \S^2$ define the $4$-manifold $M_c$ by attaching a 2-handle to the boundary of $\S^1 \times \D^3$ along $c$, see Figure~\ref{fig:Kirk} for an example. The group $\pi\coloneqq\pi_1M_c$ is infinite cyclic with generator $t$ represented by the meridian of the dotted circle. The inclusions of $t\cdot\Z[t]$ into  $\Z[\pi\sm 1]^\sigma$ and into $\Z[\pi\sm 1]/ \langle \ol{r} - r \rangle $ are both isomorphisms, so we identify the targets of $\md$ and $\mu_2$ with $t\cdot\Z[t]$.

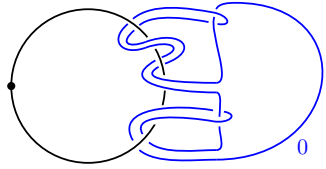
\begin{figure}[!htbp]
    \centering
    \begin{tikzpicture}[scale=1,every node/.style={scale=0.9},even odd rule,line width=0.8pt]
    \clip (-1.4,-1.45) rectangle (3.8,1.55);
        \draw (0,0) circle (1.2cm);
        \fill (-1.2,0) circle (1.8pt);
        \fill[white] 
            (1,-0.6) rectangle (1.2,-0.3) (1.1,-0.07) rectangle (1.3,0.17) (1,0.68) circle (3.2pt);
        \draw[blue]
            (3.35,-0.95) node{$0$}
            (2,-1.15) to[out=0,in=0,distance=2cm] (2.3,1.3) 
            to[out=180,in=70,distance=0.1cm] (2,1.2)
            (2,-1)  to[out=0,in=-90,distance=0.15cm] (2.05,-0.63)
            (2.05,-0.48) to[out=90,in=0,distance=0.1cm] (2,-0.1)
            (2,0.05) to[out=0,in=-110,distance=0.25cm] (1.97,1.08);
        \draw[blue] 
            (2,-1.15) to[out=180,in=-40,distance=0.35cm] (0.8,-1)
            (2,-1) to[out=180,in=-40,distance=0.3cm] (0.95,-0.88);
        \draw[blue]
            (0.7,-0.85) to[out=120,in=170,distance=0.6cm]  (1.95,-0.38)
            (0.85,-0.75) to[out=120,in=170,distance=0.4cm]  (1.95,-0.55);
        \draw[blue]
            (1.95,-0.55) to[out=-10,in=-10,distance=0.25cm]  (2.12,-0.4);
        \draw[blue] 
            (2,-0.1) to[out=180,in=-150,distance=0.7cm] (1.05,0.4)
            (2,0.05) to[out=180,in=-150,distance=0.45cm] (1.1,0.3);
        \draw[blue] 
            (1.15,0.45) to[out=20,in=20,distance=0.4cm] (0.9,0.6)
            (1.2,0.35) to[out=20,in=20,distance=0.6cm] (0.9,0.69);
        \draw[blue] 
            (0.9,0.6) to[out=-160,in=-140,distance=0.4cm] (0.62,0.97)
            (0.9,0.69) to[out=-160,in=-140,distance=0.3cm] (0.72,0.87);
        \draw[blue] 
            (1.95,1.15) to[out=170,in=40,distance=0.4cm] (0.7,1.05)
            (1.88,1) to[out=170,in=40,distance=0.3cm] (0.8,0.95);
        \draw[blue] 
            (1.95,1.15) to[out=-10,in=-10,distance=0.25cm] (2.02,0.97);
\end{tikzpicture}
    \caption{A handle diagram for $M_c$ with $\mu_2(S_c)=-t+t^2$.}
    \label{fig:Kirk}
\end{figure}
Let $S_c\in\pi_2M_c$ be represented by the sphere built out of the core of the 2-handle, together with a null homotopy of $c$ in $\S^1 \times \S^2$. One can compute $\mu_2(S_c)\in t\cdot\Z[t]$ by undoing $c$ via self-intersections and summing a term $\pm t^\ell$ for each double point, where $\ell$ is the absolute value of the linking number of the double point loop with the dotted circle. The following provides many algebraic possibilities for the image of $\md$; the proof is after Example~\ref{ex:closed-dax}.
\begin{lemma}\label{lem:M-c}
    For $M_c$ as above, the subgroup $\md(\pi_3M_c)$ is equal to the principal ideal $\mu_2(S_c)\cdot \Z[t]\subseteq t\cdot \Z[t]$.
    Moreover, for any polynomial $f\in t\cdot \Z[t]$ there exists a framed null homotopic knot $c\colon\S^1 \hra \S^1 \times \S^2$ with $\mu_2(S_c)=f$. 
    
    Therefore, any group $t\cdot \Z[t]/(f)$ arises as the target $\RGR^\sigma/\md(\pi_3M_c)$ of the invariant $\Dax$ for some $c$.
\end{lemma}
For example, $t\cdot \Z[t]/(n t^m) \cong \Z^{m-1} \times (\Z/n)^\infty$ arises. The group $t\cdot \Z[t]/(f)$ is finitely generated if and only if the leading coefficient of $f$ is $\pm 1$. By attaching several 2-handles one can see that any finitely generated abelian group can be realized as the target of the $\Dax$ invariant. 

To put this into the context of Theorems~\ref{thm-intro:Dax-classifies} and \ref{thm-intro:Dax-values}, we need a boundary condition $k\colon\S^1\hra \partial M$ with a dual sphere $G$. For $M=M_c$ these exist only if $c$ is a 0-framed unknot, giving exactly Example~\ref{ex:boundary-conn-sum}. However, in the boundary connected sum $M=(\D^2 \times \S^2)\, \natural\, M_c$ we can take $k=\S^1 \times pt$ and $G=0 \times \S^2$, and Theorem~\ref{thm:Hopf} implies $\md(\pi_3M) = \md(\pi_3M_c)= \mu_2(S_c)\cdot \Z[t]$.

\subsection{Nilpotent group structure on neat disks}\label{sec-intro:groups}
A surprising by-product of our methods is a group structure on the set of isotopy classes $\Dk$  in Theorem~\ref{thm-intro:Dax-values}. This will enable many calculations and be related to the mapping class group $\pi_0\Diff_\partial(M)$, see Sections~\ref{sec:mcg} and \ref{sec:dax-proofs}.

The group structure on $\Dk$ actually exists at the space level: The space of embeddings $\Emb(\D^2,M;k)$ is homotopy equivalent to a based loop space $\Omega E$ and the new group structure is the one of the fundamental group $\pi_1E$. This is also the key to our proofs and will be discussed in Section~\ref{sec:space-level}.

\begin{figure}[!htbp]
    \centering
    \includegraphics[width=1\linewidth]{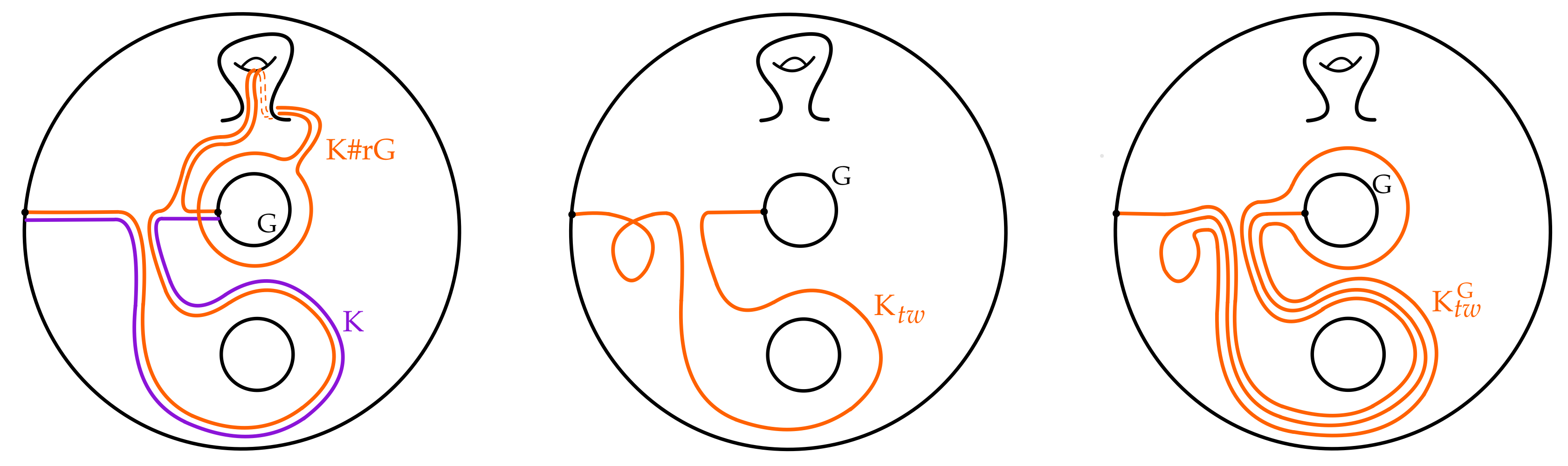}
    \caption{The interior connected sum $K \# rG$, the immersed disk $K_{tw}$ obtained by an interior twist on $K$, and the resolved disk $K_{tw}^G$.}
    \label{fig:K-tw}
\end{figure}
To understand the group structure on $\Dk$, we extend our $\RGR$-action on $\Dk$ to a $\Z[\pi]$-action.
Let $1 \in\pi$ act by sending $K$ to the disk $K+\fm(1)^G\coloneqq K_{tw}^G$ obtained from $K$ by adding one positive \emph{interior twist} as in Figure~\ref{fig:interior-twist} and tubing the resulting double point into $-G$; see Figure~\ref{fig:K-tw} for a dimension reduced picture (arcs in a surface) and Section~\ref{sec:actions} for details. 
The orientations of $G$ and $k$ (and hence $K$) are chosen so that the unique intersection point $k\cap G$ is positive. 

For any $\U\in\Dk$, taking the Euler number $e_\U(K)$ of the normal bundle $\nu K$ relative to the framing on $k$ induced by $\nu\U$, gives a map
\begin{equation}\label{eq:e-U}
    e_\U\colon\Dk\ra\Z,\quad K\mapsto e_\U(K).
\end{equation}
A regular homotopy, like a finger move, does not change this framing information, however, an interior twist changes it by 2. Since $G$ is framed, it follows that $e_\U(\U+\fm(r)^G)$ is twice the coefficient of 1 in $r\in\Z[\pi]$. The decomposition $\Z[\pi]=\RGR \times \Z\cdot 1$ allows us to take the product of a Dax invariant with $e_\U/2$ to get the following main result.
\begin{theorem}\label{thm-intro:groups}
    If $k\colon\S^1\hra\partial M$ has a dual $G\colon\S^2\to\partial M$ then any choice of $\U\in\Dk$ gives a group structure on the set $\Dk$ that fits into a group extension (by two abelian groups)
    \[\begin{tikzcd}
    \faktor{\Z[\pi]}{\md(\pi_3M)} \arrow[tail,shift left]{rr}{\U+\fm(-)^G} &&
    \arrow[dashed,shift left]{ll}{\Dax\ \times \ e_\U/2}  \Dk \arrow[two heads]{r}{p_\U} &
    \faktor{\pi_2M}{\Z[\pi]\cdot G} 
    \end{tikzcd}
    \]
    In particular, $\U=\U+\fm(0)^G$ is the trivial group element and the dotted splitting is defined on the kernel of the epimorphism $p_\U(K)\coloneqq[-\U\cup K]$. 
    Up to isomorphism of extensions, this structure does not depend on $\U$. 
\end{theorem}
In \cite[Prop.4.2]{Gabai-disks} an abelian group structure on $j^{-1}[\U]$ with unit $\U$ was introduced such that $\Dax(\U,\bull)\colon j^{-1}[\U]\sra \RGR^\sigma/\md(\pi_3M)$ is an epimorphism. Our Theorem~\ref{thm-intro:groups} not only implies that this is an isomorphism of groups, via Proposition~\ref{prop:group-diagram}, but also that the entire set $\Dk$ can be given a group structure -- that can be nonabelian, as we shall see in Proposition~\ref{prop:commutators}.
In Theorem~\ref{thm-intro:groups} we are using a more general $\Dax$ invariant, derived from that for half-disks in Section~\ref{sec:Dax-half}.

By \cite[Lem.5.14]{KT-highd} $e_\U$ is a homomorphism; its image is in $2\cdot\Z$ if and only if the universal cover $\wt M$ is spin. However, a spin$^c$ structure $W$ on $\wt M$, that always exists, will lead in Proposition~\ref{prop:splitting} to a homomorphism $\eta_{W,\U}\colon\Dk\to\Z$ that also takes the value 1 on $\U_{tw}^G$. In particular, the group $\Dk$ always splits off 
a factor of $\Z$.
\begin{remark}
    \label{rem:independence}
    If $G\colon\S^2\hra\partial M$ has dual circles $k$ and $k'$ that are null homotopic in $M$, one gets an isomorphism $\Dk\cong \D(M,k')$, a straightforward consequence of our construction. 
\end{remark}

\begin{example}\label{ex:boundary-conn-sum}
    For the boundary connected sum  $M=(\D^2 \times \S^2) \natural (\S^1 \times \D^3)$ with $\U\coloneqq \D^2 \times pt$ and $G\coloneqq 0 \times \S^2$, we have $\md=0$, see \cite[Thm.3.10]{Gabai-disks} or apply Theorem~\ref{thm:Hopf}. Since $M\simeq \S^2 \vee \S^1$ and $\pi_2M/ \langle G \rangle =0$, the extension in Theorem~\ref{thm-intro:groups} gives $\Z[t,t^{-1}] \cong \Dk$ and $\mu_2^{-1}(0)\cong\Z[t,t^{-1}]/ \langle t+t^{-1} \rangle$, for $t$ a generator of $\pi_1M$. Therefore, every homotopy class of embeddings contains infinitely many isotopy classes, which is a completely new phenomenon, very different from the spherical case. 
\end{example}


 
We next determine all group commutators in $\Dk$. Denote the \emph{intersection form} of $M$ by $\lambda\colon\pi_2M\times\pi_2M\to\Z[\pi]$ and its coefficient at $1$ by $\lambda_1$, and write $\lambdabar\coloneqq\lambda-\lambda_1$,  the \emph{reduced intersection form} of $M$.
\begin{prop}\label{prop:commutators}
    The group commutator of $K_1,K_2\in\Dk$ is
    \[
[K_1, K_2]=\U+\fm(\lambdabar(-\U\cup K_1, -\U\cup K_2))^G.
    \]
    In particular, the extension in Theorem~\ref{thm-intro:groups} is central and the group $\Dk$ is 2-step nilpotent. Moreover, $\Dk$ is abelian if and only if we have $\lambdabar(a_1,a_2)\in \md(\pi_3M)$ for all $a_1,a_2\in\pi_2M$. 
\end{prop}
The centrality simply follows from the fact that $-\U\cup\bull$ vanishes on classes coming from the left, so the $\lambdabar$-term in the commutator vanishes as well. We use the fact that the intersection form $\lambda$ of $M$ factors through the quotient $\pi_2M/\Z[\pi]\cdot G$ because $G$ lies in the boundary of $M$.

In Example~\ref{ex:boundary-conn-sum} the group $\Dk$ is free abelian of infinite rank. By Proposition~\ref{prop:commutators}, to find a nonabelian example it suffices to find classes $a_1,a_2\in\pi_2M$ such that $\lambdabar(a_1,a_2)\notin \RGR^\sigma$, simply because $\md(\pi_3M)$ lies in this fixed set. For example, take $M=M_1 \# M_2$ for any simply connected 4-manifold $M_1$ that has nontrivial intersection form, and where $\pi_1M_2$ contains an element $g$ with $g\neq \ol{g}$. Then $\lambda_{M_1}(a_1,a_2)\in \Z\sm 0$ for some $a_i\in\pi_2M_1$, so
$\lambdabar_M(g\cdot a_1, a_2)= g \cdot \lambda_{M_1}(a_1,a_2) \neq \ol{g}\cdot \lambda_{M_1}(a_1,a_2)=\ol{\lambdabar_{M}(g\cdot a_1,a_2)}$.
\begin{remark}
    The commutator pairing of the extension in Theorem~\ref{thm-intro:groups} induces a skew-symmetric map $\pi_2M \times \pi_2M\to \Z[\pi]/\md(\pi_3M)$,
    but Proposition~\ref{prop:commutators} involves the hermitian pairing $\lambdabar$ which is \emph{not} skew-symmetric for most $4$-manifolds! Fortunately, the formula \eqref{eq:Whitehead} for Whitehead products, proven in Proposition~\ref{prop:Whitehead}, says that this becomes true modulo $\md(\pi_3M)$.  
\end{remark}

\textbf{Acknowledgments.} 
We thank Dave Gabai for sharing his insight into the Dax invariant of Example~\ref{ex:boundary-conn-sum} during his visit to Bonn. We also thank Hannah Schwartz and Rob Schneiderman for interesting discussions related to this work.
Both authors cordially thank the Max Planck Institute for Mathematics in Bonn. The first author was also supported by the Fondation Sciences Math\'ematiques de Paris.

\tableofcontents

\section{Further results and relations to other work} \label{sec:further results}
We first fix some notation. Recall that $M$ is a smooth, oriented 4-manifold with a fixed knot $k\colon\S^1\hra \partial M$ in the boundary. If there exists $G\colon\S^2\hra\partial M$ with $k\pitchfork G=pt$, and a neat embedding $\U\colon\D^2\hra M$ with boundary $\partial\U=k$, we say we are in the \emph{setting with a dual} and denote it by the triple $(M,\U,G)$. 

We let $M_G$ denote the manifold obtained from $M$ by attaching a 3-handle $h^3$ to $G$. This is well defined because $G$ is uniquely framed by our choices of orientations. There is a diffeomorphism
    \begin{equation}\label{eq:MG}
M_G\coloneqq M \cup_G h^3 = \big(M \sm \nu \U\big)\cup (h^2 \cup_G h^3)  \cong M \sm \nu \U
    \end{equation}
of $M_G$ with the complement of an open tubular neighborhood $\nu\U$ of $\U$ in $M$. Namely, $M$ is obtained from $M \sm \nu \U$ by attaching a $2$-handle $h^2$ with cocore $\U$. As $G$ is dual to $k=\partial \U$ and has trivial normal bundle, the standard handle cancellation of this $2$- and $3$-handle pair gives the diffeomorphism.

This implies that the diffeomorphism type of the 4-manifold $M\sm\nu\U\cong M\sm U$ is independent of the choice of $\U$, as long as $k$ is fixed. Hence, the main source of isotopy invariants for such $\U$, namely the diffeomorphism class of its complement, is completely useless in Theorem~\ref{thm-intro:Dax-classifies}!
Furthermore, by reversing the argument, we see that in the presence of any $\U$ with boundary $k$, the diffeomorphism type of the manifold $M_G$ is actually independent both of the choice of a dual $G$ for $k$ and the framing of $G$.

Recall the map $e_\U$ from~\eqref{eq:e-U}. Its kernel is a normal subgroup 
\begin{equation}\label{eq:Dk0}
    \Dk^0\coloneqq\ker(e_\U)\;\trianglelefteq\; \Dk.
\end{equation}
consisting of disks that induce the same framing on the boundary knot $k$ as the chosen undisk $\U$.

\subsection{Application to mapping class groups of 4-manifolds} \label{sec:mcg}
In the setting with a dual as above, the complement of $\nu\U$ does not carry any isotopy information, so it is natural to expect that different isotopy classes of $\U$ will lead to interesting diffeomorphisms of the 4-manifold $M$. This is indeed the case, as the following Theorem~\ref{thm:mcg} shows.

Let $\Diff_\partial(X)$ denote the topological group of diffeomorphisms of a manifold $X$ that are identity on $\partial X$. Extending a diffeomorphism by the identity over the 3-handle $h^3$ in $M_G$ is a homomorphism $i_G\colon\Diff_\partial(M)\hra\Diff_\partial(M_G)$, and similarly extending over $\nu\U$ gives $i_\U\colon\Diff_\partial(M\sm\nu\U)\hra\Diff_\partial(M)$. In Lemma~\ref{lem:i_Gi_U} we show that the composite $i_G\circ i_\U$ is a weak homotopy equivalence and in particular, $\pi_0i_G$ is a split surjection. 
\begin{theorem}\label{thm:mcg}
    In the setting $(M,\U,G)$ there is a split extension
\[\begin{tikzcd}
       \Dk^0 \arrow[tail]{r}{a_{\U}}
        & \pi_0 \Diff_\partial(M)  \arrow[two heads]{r}{\pi_0i_G}
        & \pi_0 \Diff_\partial(M_G)
\end{tikzcd}
\]
    of groups,
    where $a_\U$ is an ``ambient isotopy'' map that, roughly speaking, takes a neat disk $K$, considers $-\U\cup K$ as a loop of embedded arcs, extends this isotopy to an ambient isotopy and takes the endpoint $a_\U(K)$.
\end{theorem}
A point-theoretic splitting of $a_{\U}$ is induced by sending a diffeomorphism $\varphi$ to the disk $\varphi(\U)$. 
In future work we will show how Theorem~\ref{thm:mcg} generalizes results of Wall on diffeomorphisms of stabilized 4-manifolds, show that the action of $\pi_0 \Diff_\partial(M_G)$ is the natural one on the normal subgroup and conclude that $\varphi\mapsto \varphi(\U)$ is rarely a group homomorphism, i.e.\ this semi-direct product is rarely a direct product of groups.
\begin{remark}\label{rem:diffeos} 
    It follows that for $K\in\Dk^0$ there is a diffeomorphism of $M$ rel.\ boundary, that takes $K$ to $\U$. 
    One can see this directly using that $M\sm \nu K\cong M_G\cong M\sm\nu\U$ as in \eqref{eq:MG} extends across $\nu K$ if the framings agree. Compare \cite[Lem.2.3]{Schwartz},
    \cite[Lem.6.1]{ST-LBT}, \cite[Lem.5.3]{Gabai-disks}.
\end{remark}

\subsection{Deriving the light bulb trick for spheres} \label{sec:spheres}
For a framed embedding $G\colon\S^2\hra N$ in the interior of a smooth oriented 4-manifold $N$, the set $\SG$ of isotopy classes of embeddings $F\colon\S^2\hra N$ that are dual to $G$, was studied previously in \cite{Gabai-spheres, ST-LBT}.
Gabai's LBT for spheres needed the assumption that the fundamental group of the ambient 4-manifold $M$ has no elements of order~2. Together with Rob Schneiderman \cite{ST-LBT}, the second author found a completely general isotopy classification of spheres with common framed dual, by introducing a ``Freedman--Quinn invariant'' $\FQ$ in addition to the homotopy class of the sphere. 

That proof shows, roughly speaking, that the Whitney moves in a homotopy can be chosen as inverses to the finger moves, up to isotopy of the two 2-spheres and assuming control over $\FQ$ (that vanishes if there are no elements of order 2 in $\pi_1M$). Remarkably, the Freedman--Quinn invariant can be thought of as the Dax invariant for spheres, despite their very different origins, see Proposition~\ref{prop:disks-spheres}.
These two previous proofs have not been generalized from spheres to disks, but see \cite{Schwartz-disks} for partial results in this direction. 

These isotopy classifications for spheres follow from our setting of disks as follows.
 Denote by $M\coloneqq N\sm\nu G\subseteq N$ the complement of an open tubular neighborhood of $G$, and by $m_G\colon\D^2\hra N$ a meridian disk of $G$. Then the boundary $\partial M$ contains the meridian circle $k\coloneqq\partial m_G$ and also a dual sphere for $k$, a parallel push-off of $G$. Thus, we are in a setting $(M,\U,G)$ as in Proposition~\ref{prop-intro:4-term} and we can classify the set $\Dk$ of disks in $M$ with boundary $k$ via an exact sequence. We next use this to describe $\SG$ itself.

Forgetting that $F$ is an embedding gives the map $j$ to the set $[\S^2,N]^G$ of based homotopy classes of maps $S\colon\S^2\to N$ with $\lambda(S,G)=1$. Note that neither $\SG$ nor $[\S^2,N]^G$ comes with a group structure, but \cite{ST-LBT} defines an action $\fm(\bull)^G$ on $\SG$ by a quotient of the $\mathbb{F}_2$-vector space generated by 2-torsion $T_N$ in $\pi_1N$.
\begin{prop}\label{prop:disks-spheres}
    There is an isomorphism of 4-term exact sequences
    \[\begin{tikzcd}[column sep=0.78cm]
        \faktor{\Z[\pi\sm 1]^\sigma}{\md(\pi_3M)}\arrow{d}[swap]{\cong} \arrow[tail,squiggly,shift left]{r}{+\fm^G} 
        & \Dk \arrow[dashed,shift left]{l}{\Dax} \arrow{d}{\bull\cup m_G}[swap]{\cong} \arrow{r}{j} 
        & \mbox{}[\D^2,M;k] \arrow{d}{\bull\cup m_G}[swap]{\cong}\arrow[two heads]{r}{\mu_2} 
        & \faktor{\RGR}{\langle \ol{r}{-}r\rangle}\arrow{d}[swap]{=}\\
        \faktor{\mathbb{F}_2[T_N]}{\mu_3(\pi_3N)} \arrow[tail,squiggly,shift left]{r}{+\fm^G} 
        & \SG \arrow[dashed,shift left]{l}{\FQ}  \arrow{r}{j} 
        & \mbox{}[\S^2,N]^G \arrow[two heads]{r}{\mu_2} 
        & \faktor{\Z[\pi_1N\sm 1]}{\langle\ol{r}{-}r\rangle}
    \end{tikzcd}
    \]
    The relative Dax respectively Freedman--Quinn invariants $\Dax$ and $\FQ$ invert the actions as in Theorem~\ref{thm-intro:Dax-values}~
    (3) and they make the left square commute. 
\end{prop}
The relative Freedman--Quinn invariant $\FQ$ was defined and shown to invert the action in \cite{ST-LBT}, together with the exactness of the lower sequence. We will show in Section~\ref{sec:spheres-proofs} how those results follow from our upper sequence. The crucial step was sketched in Example~\ref{ex:interior-conn-sum}, which applies to the manifold $M=N\sm\nu G$ since $\partial M\cong\S^1 \times \S^2\sqcup\partial N$ and $\S^1\times pt$ bounds in $M$, implying that $M\cong(\D^2\times\S^2) \# M'$. Thus, $g+\ol{g}$ is contained in $\md(\pi_3M)$ for all $g$, which gives rise to the leftmost isomorphism since $\Z[\pi]^\sigma/\langle1,g+\ol{g}\rangle\cong\mathbb{F}_2[T_N]$.

\begin{remark}\label{rem:deriving-disks-from-spheres}
    Note that a triple $(M,\U,G)$ is of such a connected sum form if and only if $G$ lies in a connected component of $\partial M$ that is diffeomorphic to $\S^1 \times \S^2$. In this situation, one can forget the rest of the boundary $\partial M$ and prove our result for disks with the methods of \cite{ST-LBT} by first filling in $\D^2 \times \S^2$, i.e.\ turning $\U$ into a sphere in the interior, and again using Example~\ref{ex:interior-conn-sum} to see that the classifying invariants, Dax and Freedman--Quinn, agree.
\end{remark}

In Section~\ref{subsec:Dax-FQ} we will see that the $\Dax$ invariant always determines the Freedman--Quinn invariant for disks, and note that if the component of $\partial M$ containing $G$ is not diffeomorphic to $\S^1 \times \S^2$, then $\Dax$ is a strictly stronger invariant in general, taking values in a larger abelian group (torsion higher than just 2-torsion and also torsion-free parts arise, see Lemma~\ref{lem:M-c}). 

For example, $M=(\D^2\times\S^2)\#(\S^1 \times \D^3)$ is a non-simply connected example for which $\RGR^\sigma/\md(\pi_3M)$ is trivial, so in $M$ there is at most one isotopy class of embeddings in a given homotopy class. In contrast, we saw in Example~\ref{ex:boundary-conn-sum} that the \emph{boundary} connected sum $(\D^2 \times \S^2)\natural(\S^1 \times \D^3)$ has an interesting $\Dax$ invariant; this was utilized in \cite[Thm.0.8]{Gabai-disks} to construct in this manifold a 3-ball homotopic but not isotopic to $pt \times \D^3$.
\begin{remark}\label{rem:based}
    In both \cite{Gabai-spheres} and \cite{ST-LBT} it is only assumed that the given spheres with dual $G$ are homotopic, not \emph{based} homotopic. However, in both cases it is argued at the start of the proofs why such a homotopy can also be found in the based setting, see \cite[Lem.2.1]{ST-LBT} for an elementary argument.
\end{remark}

\subsection{Key inputs from  spaces of embeddings} \label{sec:space-level}
Crucial ingredients for the above results come from our previous work \cite{KT-highd}, concerning neat embeddings of $k$-disks into $d$-manifolds, with a dual sphere for the boundary condition. The following is the case $k=2,d=4$ of {\cite[Thm.B]{KT-highd}}.
\begin{theorem}
\label{thm:space-LBT}
    In the setting $(M,k,G)$, any choice of $\U\in\Emb(\D^2,M;k)$ induces homotopy equivalences
    \[\Emb(\D^2,M;k) \simeq \Emb(\HD,M_G;k^\e)\simeq \Omega\Emb^\e(\D^1,M_G;k_0^\e).
    \]
\end{theorem}
The space $\Emb^\e(\D^1,M_G;k_0^\e)$ of ``$\e$-augmented neat arcs'' appearing on the right consists of embeddings $\D^1 \times [0,\e]\hra M_G$ with a fixed boundary condition  $k_0^\e\colon ([-1,-1+\e]\sqcup [1-\e,1]) \times [0,\e]\hra \partial M_G$. Up to homotopy equivalence, this corresponds to a neat arc with a choice of normal vector field, see \eqref{eq:eta-W-pr} or \cite[Prop.5.1]{KT-highd}.
We pick an $\e$-augmented arc $u^\e$ as a basepoint and write $\Omega\Emb^\e(\D^1,M_G;k_0^\e)$ for the space of loops based at it.

The proof in \cite{KT-highd} that this loop space and $\Emb(\D^2,M;k)$ are homotopy equivalent uses an intermediate space $\Emb(\HD,M_G;k^\e)$ of \emph{half-disks} in~$M_G$. These are embeddings of disks whose boundary $k=k_-\cup u$ consists of an arc $k_-\colon\D^1\hra \partial M_G$ and a neat arc $u\colon\D^1\hra M_G$, so that $\partial k_-=k_0=\partial u$. Such a $k$ in $M_G$ arises precisely as $k\colon\S^1\hra \partial M \subset M\subset M_G\coloneqq M \cup_{\nu G} h^3$, since the arc $u\coloneqq k\cap\nu G$ is neat in $M_G$, whereas $k_-\coloneqq k\cap(\partial M\sm\nu G)$ remains in the boundary after the handle $h^3$ is attached. As a consequence, a disk in $M$ becomes precisely a half-disk in $M_G$. Moreover, we require that half-disks in $\Emb(\HD,M_G;k^\e)$ all agree actually on a collar $k^\e$ of their boundary.

To prove that $\Emb(\HD,M_G;k^\e)$ deloops to the space of $\e$-augmented arcs, we use a trick going back to Cerf (in the case of diffeomorphisms of $\D^d$): the larger space $\Emb(\HD,M_G;k_-)$ is contractible, and the restriction to the free boundary arc is a fibration, so a connecting map is a desired equivalence. 

Taking connected components of the spaces in Theorem~\ref{thm:space-LBT} gives bijections
\begin{equation}\label{eq:bijections}
    \Dk \longleftrightarrow \HD(M_G;k^\e) \longleftrightarrow\pi_1(\Emb^\e(\D^1,M_G;k_0),u^\e).
\end{equation}
between the sets of isotopy classes of neat disks in $M$ and half-disks in $M_G$, and the fundamental group of the space of $\e$-augmented arcs in $M_G$.
This provides the explanation for the group structure on $\Dk$ from Section~\ref{sec-intro:groups}: it comes by definition from the concatenation of loops of arcs.

In \cite[Sec.5]{KT-highd} we next studied the map forgetting the augmentation $\ev^\e\colon\Emb^\e(\D^1,M_G;k_0^\e) \sra \Emb(\D^1,M_G;k_0)$.
This is a fibration, whose fiber $\Omega\S^2$ measures the normal derivative along an arc. We showed that $\pi_1\ev^\e$ is onto with kernel $\Z$, and that this $\Z$ can be split back using a splitting $\eta_W$ (which depends on the choice of a spin$^c$ structure $W$ on the universal cover of $M_G$). Thus, there is a group isomorphism:
\begin{equation}\label{eq:eta-W-pr}
    \eta_W\times\ev^\e\colon
    \pi_1(\Emb^\e(\D^1,M_G;k_0),u^\e)\xrightarrow{\cong} 
    \Z \times \pi_1(\Emb(\D^1,M_G;k_0),u).
\end{equation}
The fundamental group of the space of arcs in a 4-manifold falls into the so-called metastable range, so it can be computed using the work of Dax~\cite{Dax} (following Haefliger), or at the second stage of the Goodwillie--Weiss embedding calculus, see \cite{GKW}. We make Dax's result explicit and computable, and for this case obtain the following, see~\cite[Thm.4.13, Thm.4.20]{KT-highd}.
\begin{theorem}\label{thm:Dax-arcs}
    Fix an oriented smooth $4$-manifold $X$ and $k_0\colon\S^0\hra\partial X$. Then there is an inverse pair of isomorphisms
    \[\begin{tikzcd}
        \Da\colon \pi_2\big(\Imm(\D^1,X;k_0),\Emb(\D^1,X;k_0),u\big)\rar[shift left] & \Z[\pi_1X]\colon\lar[shift left] \realmap
    \end{tikzcd}
    \]
\end{theorem}
This is our second black box, which we use in Section~\ref{sec:arcs} to describe the group $\pi_1(\Emb(\D^1,M_G;k_0),u)$ as an extension of $\pi_2(M_G)$ by a certain quotient of $\Z[\pi_1M_G\sm1]$, and study its commutator pairing. We use~\eqref{eq:eta-W-pr} to extend this to the augmented version in Proposition~\ref{prop:splitting}.  In Section~\ref{sec:main-proofs} this will give rise to our main results, using \eqref{eq:bijections} and our action $+\fm^G$, described in details in Section~\ref{sec:actions}. The main step, Proposition~\ref{prop:Dax-inverts}, says that the relative $\Dax$ invariant inverts this action. 

\subsubsection*{Notation} 
Throughout the paper we assume $X,N$ are smooth, oriented, connected 4-manifolds, with $\partial X\neq\emptyset$, and $N$ has no condition on the boundary but comes with a fixed framed $G\colon\S^2\hra N$. See~\nameref{sec:notation}.

\section{Spaces of embedded and immersed arcs}\label{sec:arcs}
We review some notions from \cite{KT-highd} in the special case of arcs in 4-manifolds and prove a number of results that are characteristic for this case. For our 4-manifold $X$ with $\pi\coloneqq\pi_1(X,x_-)$ and $u\colon\D^1\hra X$, we consider the spaces $\Emb_\partial(\D^1,X)\coloneqq\Emb(\D^1,X;k_0)$ and $\Imm_\partial(\D^1,X)$ of neat embedded and immersed arcs, with boundary $k_0=\partial u=\{x_-,x_+\}$ and based at $u$.


\subsection{The Dax invariant for spaces of arcs}\label{sec:Dax-arcs}
We now recall the definition of the isomorphism $\Da$ appearing in Theorem~\ref{thm:Dax-arcs}, since we need to identify this invariant for arcs with the relative $\Dax$ invariant for disks used in our main results. The domain of $\Da$ depends on the basepoint $u$, but we will see that this dependence disappears if $u$ is homotopic into the boundary, as in our case. For an explicit inverse of $\Da$, the \emph{realization map} $\realmap$, see \cite[Sec.4.1.4]{KT-highd}.
We point out that \cite{Gabai-disks} proves statements equivalent to this isomorphism using the ``spinning map'' and the same definition of $\Da$.

To define $\Da$ for arcs, denote $\I\coloneqq[0,1]$ and represent a given relative homotopy class by a map $F\colon \I^2\to \Imm_\partial(\D^1,X)$ satisfying
\[
    F( \partial\I \times \I \cup  \I \times 1 ) =u
    \quad\text{ and }\quad  F( \I \times 0 )\subseteq\Emb_\partial(\D^1,X).
\]
After a small perturbation of $F$ preserving boundary conditions the map
\[
    \wt{F}\colon\;\I^2 \times \D^1 \ra \I^2 \times X,\quad (\vec{t},\theta)
    \mapsto (\,\vec{t},\, F(\vec{t}\ )(\theta)\,)
\]
is an immersion whose only singularities are isolated transverse double points in the interior of $\I^2 \times \D^1$.
By definition, double point values have the form $(\vec{t}_i, y_i)$ for $\vec{t}_i\in\I^2$ and $y_i\coloneqq F(\vec{t}_i)(\theta^-_i)=F(\vec{t}_i)(\theta^+_i)\in X$ with $\theta^-_i<\theta^+_i\in\D^1$.
\begin{defn}\label{def:Dax} 
  Let $\Da[F]\coloneqq\sum_{i=1}^k\e_{(\vec{t}_i, y_i)}g_{y_i}\in\Z[\pi]$ with signs and  group elements as follows.
 Let $\e_{(\vec{t}_i, y_i)}\in \{\pm 1\}$ be the relative orientation at $(\vec{t}_i, y_i)$, obtained by comparing orientations of the tangent space $T_{(\vec{t}_i, y_i)}(\I^2 \times X)$ and:
\begin{equation}\label{eq:Dax-sign}
        d\wt{F} \big( T_{(\vec{t}_i,\theta_i^-)}( \I^2 \times \D^1 ) \big)
        \oplus d\wt{F} \big( T_{(\vec{t}_i,\theta_i^+)}(\I^2 \times \D^1) \big).
\end{equation}
    Define the group element $g_{y_i}\in\pi_1(X,x_-)$ to be represented by the following loop based at $u(-1)=x_-$ (see Figure~\ref{fig:double-point}):
\begin{equation}\label{eq:Dax-dp-loop}
      F(\vec{t}_i)|_{[-1,\theta_i^-]}\cdot F(\vec{t}_i)|^{-1}_{[-1,\theta_i^+]}.
\end{equation}
\end{defn}
Note how the order $\theta^-_i<\theta^+_i$ along the arc is crucial for defining these loops. In contrast, when computing associated loop for a self-intersection of an immersion $\D^3\imra Y^6$, one has the indeterminacy $g_y = (-1)g_y^{-1}$ coming from the choice of the order of sheets at $y$ in the 6-manifold $Y$, see Section~\ref{subsec:Dax-FQ}.
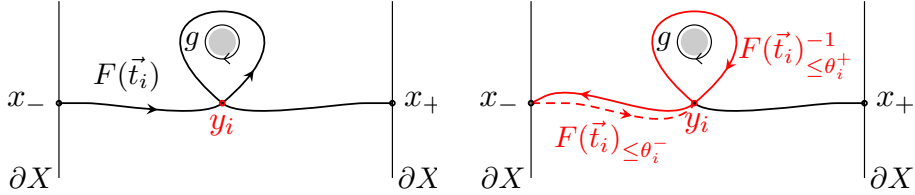
\begin{figure}[!htbp]
    \centering
        \begin{tikzpicture}[scale=0.8,every node/.style={scale=1},even odd rule]
        \clip (-0.75,-1.3) rectangle (5.55,1.6);
        \draw 
            (0,1.5) -- (0,-1.1) node[left=-2pt]{\small$\partial X$}  
            (4.9,1.5) -- (4.9,-1.1)  node[right=-2pt]{\small$\partial X$};
        \draw[black,semithick,postaction=decorate,decoration = {markings, mark = at position 0.6 with {\arrow{stealth}}}]
            (0,0) -- 
            (0.4,0) to[out=0, in=-140, distance=0.4cm]  (2.4,0)
            (1,0) node[above]{\small$F(\vec{t}_i)$};
        \draw[black,semithick]
            (2.4,0) to[out=-40,in=180, distance=0.4cm] 
            (4.5,0) -- (4.9,0)
            ;
        \draw[black,semithick,postaction=decorate,decoration = {markings, mark = at position 0.17 with {\arrow{stealth}}}]
            (2.4,0)
            to[out=40,in=140, distance=2.8cm]
            (2.4,0) ;
        \draw[red,semithick]
            (2.4,0) circle (1pt) node[below]{$y_i$};
        \draw[black,semithick]
            (0,0) circle (1pt) node[left]{$x_-$} 
            (4.9,0)  circle (1pt) node[right]{$x_+$};
            
        \fill[color=black!20!white] 
            (2.4,0.9) circle (0.2cm);
        \draw[black,postaction = decorate,decoration = {markings, mark = at position 0.76 with {\arrowreversed{>}} }] 
            (2.4,0.9) circle (0.25cm) node[left=0.13cm]{\small$g$};
    \end{tikzpicture}~\quad
       \begin{tikzpicture}[scale=0.8,every node/.style={scale=1},even odd rule]
        \clip (-0.75,-1.3) rectangle (5.6,1.6);
        \draw 
            (0,1.5) -- (0,-1.1) node[left=-2pt]{\small$\partial X$}  
            (4.9,1.5) -- (4.9,-1.1)  node[right=-2pt]{\small$\partial X$};
        \draw[red,densely dashed,semithick,postaction=decorate,decoration = {markings, mark = at position 0.6 with {\arrow{stealth}}}]
            (0,0) -- 
            (0.5,0)
            to[out=0, in=-140, distance=0.6cm]  (2.35,-0.08)      --(2.4,0)
            (1.2,-0.08) node[below]{\small$F(\vec{t}_i)_{\leq \theta_i^-}$};
        \draw[red,semithick,postaction=decorate,decoration = {markings, mark = at position 0.3 with {\arrowreversed{stealth}}}]
            (0,0) -- (0.2,0.1) to[out=25, in=-140, distance=0.6cm]  (2.4,0)
            (3.9,0.7) node[]{\small$F(\vec{t}_i)_{\leq \theta_i^+}^{-1}$};
        \draw[black,semithick]
            (2.4,0) to[out=-40,in=180, distance=0.4cm] 
            (4.5,0) -- (4.9,0);
        \draw[red,semithick,postaction=decorate,decoration = {markings, mark = at position 0.17 with {\arrowreversed{stealth}}}]
            (2.4,0) 
            to[out=40,in=140, distance=2.8cm]
            (2.4,0) ;
        \draw[black,semithick]
            (0,0) circle (1pt) node[left]{$x_-$} 
            (4.9,0)  circle (1pt) node[right]{$x_+$}
            ;
        \draw[red,semithick,fill=black]
            (2.4,0) circle (1pt) (2.44,-0.3) node[]{$y_i$};
            
        \fill[color=black!20!white] 
            (2.4,0.9) circle (0.2cm);
        \draw[black,postaction = decorate,decoration = {markings, mark = at position 0.76 with {\arrowreversed{>}} }] 
            (2.4,0.9) circle (0.25cm) node[left=0.13cm]{\small$g$};
    \end{tikzpicture}
    \caption{The double point $y_i\in X$ of the arc $F(\vec{t}_i)\in\Imm_\partial(\D^1,X)$ has the associated loop $g_{y_i}=g^{-1}$.}
    \label{fig:double-point}
\end{figure}

In \cite{KT-highd} we then consider $(\Imm,\Emb)\coloneqq(\Imm_\partial(\D^1,X),\Emb_\partial(\D^1,X))$ with basepoint $u$, and the associated exact sequence on homotopy groups:
\begin{equation}\label{eq:pair-seq}
    \begin{tikzcd}[column sep=12pt]
    \pi_3\S^3\cong\Z \dar[tail]{i_*} &&  &&\\
    \pi_2(\Imm,u)\arrow{rr}{\delta_{\Imm}} \dar[two heads]{p_u} & &
       \pi_2(\Imm,\Emb,u) \rar{\partial}\dar[shift left=5pt]{\Da}[swap]{\cong} &
       \pi_1(\Emb,u) \rar[two heads] &
       \pi_1(\Imm,u) \dar{p_u}[swap]{\cong}\\
    \pi_3X && \Z[\pi]\uar[shift left=5pt]{\realmap} & & \pi_2X
    \end{tikzcd}
\end{equation}
The left vertical short exact sequence and the isomorphism on the right are obtained by Smale--Hirsch immersion theory from the maps 
\[
    \pi_n(\Imm_\partial(\D^1,X),u)\overset{j_*}{\ra}\pi_n(\Map_\partial(\D^1,X),u) \overset{p_u}{\ra} \pi_{n+1}(X,x_-)
\]
given by inclusion and adjunction. We showed that $\Da\circ\delta_{\Imm}\circ i_*\colon\Z\to\Z[\pi]$ sends 1 to 1 in \cite[Prop.4.28]{KT-highd}, leading to the following definition.
\begin{defn}\label{def:dax} 
    The homomorphism $\md_u\colon\pi_3X \to \RGR$ is obtained from the composite $\Da\circ\delta_{\Imm}$ by the canonical identification $\Z[\pi\sm 1] \cong \Z[\pi]/\langle 1 \rangle$.

    In other words, for $a\in \pi_3(X,x_-)$ we pick any map $S_a\colon \I^2 \to \Imm_\partial(\D^1,X)$ taking the entire $\partial \I^2$ to $u$ and such that the adjoint of the union $S_0\cup_{\partial\I^2}S_a$ is a representative of $a$, where $S_0(\vec{t})=u$ for all $\vec{t}\in\I^2$. Then $\md_u(a)=\Da(S_a)$ modulo the coefficient at $1$.
\end{defn}

In particular, the Dax invariant induces an isomorphism
\begin{equation}\label{eq:Da-final}
    \begin{tikzcd}
    \Da\colon\:
    \ker\big(p_u\colon\pi_1(\Emb,u)\sra\pi_2X\big)
    \rar{\cong} & \faktor{\RGR}{\md_u(\pi_3X)}.
    \end{tikzcd}
\end{equation}
whose inverse is the mentioned realization map $\realmap$ from \cite[Sec.4.1.4]{KT-highd}.

Next consider all maps $\Map\coloneqq\Map_\partial(\D^1,X)$ of arcs into $X$ with boundary $k_0$ and compare the pair sequence \eqref{eq:pair-seq} with that for the pair $(\Map,\Emb)$:
\begin{equation}\label{eq:Map-seq}
    \begin{tikzcd}[column sep=12pt]
  \pi_3\S^3\cong\Z \dar[tail]{i_*} \arrow{rr}{1\ \mapsto \ 1} && \Z[\pi]\dar[shift left=5pt]{\realmap} & & \\
    \pi_2(\Imm,u)\arrow{rr}{\delta_{\Imm}} \dar[two heads]{j_*} & &
       \pi_2(\Imm,\Emb,u) \dar[two heads] \rar{\partial}\uar[shift left=5pt]{\Da}[swap]{\cong} &
       \pi_1(\Emb,u)\dar[equal] \rar[two heads] &
       \pi_1(\Imm,u) \dar{j_*}[swap]{\cong}\\
     \pi_2(\Map,u)\arrow{rr}{\delta_{\Map}} \dar{p_u}[swap]{\cong} & &
       \pi_2(\Map,\Emb,u) \rar{\partial} \dar[shift left=5pt]{\Da}[swap]{\cong}&
       \pi_1(\Emb,u) \rar[two heads] &
       \pi_1(\Map,u) \dar{p_u}[swap]{\cong}\\  
    \pi_3X \arrow{rr}{\md_u} && \RGR \uar[shift left=5pt]{\realmap} & & \pi_2X
    \end{tikzcd}
\end{equation}
The advantage of working with the lower pair sequence is that we can use the $\Da$ homomorphism (with values in $\RGR$, i.e.\ not counting double points with trivial group element) without having to worry about arcs being immersed. For example, $\md_u=\Da\circ\delta_{\Map}\circ p_u^{-1}$ is computed as explained in the paragraph below Definition~\ref{def:dax}, with the only change that we only need to pick $S_a\colon\I^2\to \Map_\partial(\D^1,X)$. We still perturb the corresponding map $\wt S(\vec{t},\theta)=(\vec{t},S_a(\vec{t})(\theta))$ to be a generic immersion, but then only count its double points with nontrivial group element, such that cusps do not matter.

It is clear that $\md_u$ only depends on the homotopy (=isotopy) class of the neat arc $u$. In particular, if $u$ is homotopic rel.\ endpoints into $\partial X$,
we have
\begin{equation}\label{eq:dax-def}
    \md\coloneqq\md_u\colon\pi_3X\to\Z[\pi\sm1]
\end{equation}
which is a homomorphism
depending only on the component of $\partial X$ that contains the endpoints of $u$. 
Precisely this case occurs in our setting $X=M_G$ of Section~\ref{sec:space-level}, since there exists $\D^2\to M_G$ with boundary $k=k_-\cup u$. 

In Corollary~\ref{cor:Dax-FQ} we will show that the image of $\md$ lies in the subgroup $\Z[\pi\sm1]^\sigma$ of invariants under the involution $\sigma(g)=\ol{g}=g^{-1}$.


\subsection{Group commutators and Whitehead products}\label{sec:dax-properties}
The exact sequence \eqref{eq:Map-seq} actually reduces to the following group extension. Recall that
\begin{equation}\label{eq:lambdabar}
    \lambdabar\colon\pi_2X\times\pi_2X\to\Z[\pi\sm1]
\end{equation}
is the reduced intersection form $\lambdabar\coloneqq\lambda-\lambda_1$, where $\lambda$ counts the intersection points of two maps $\S^2\to X$ (which can be assumed to be in general position) together with signs and group elements, cf.\ Definition~\ref{def:mu_3}. To get the reduced form $\lambdabar$ we subtract the coefficient $\lambda_1$ next to the trivial group element.
\begin{prop}\label{prop:central}
    There is a central extension of groups
\[\begin{tikzcd}
        \faktor{\Z[\pi\sm1]}{\md_u(\pi_3X)} \arrow[tail]{r}{\partial\realmap}  & \pi_1(\Emb_\partial(\D^1,X),u) \arrow[two heads]{r}{p_u} &
             (\pi_2X,+)
\end{tikzcd}
\]
    with the commutator of $\ul{f_i} \in \pi_1(\Emb_\partial(\D^1,X),u)$ given by the formula
\[
    [\ul{f_1},\ul{f_2}]= \partial\realmap\big(\lambdabar(p_u\ul{f_1},p_u\ul{f_2})\big).
\]
    As a consequence, $\pi_1(\Emb_\partial(\D^1,X),u)$ is abelian if and only if the image of $\lambdabar$ is contained in $\md_u(\pi_3X)\subseteq\Z[\pi\sm 1]$. 
\end{prop}
Note that in higher ambient dimensions, the analogue of this group extension takes place in the category of abelian (higher homotopy) groups, so the commutator pairing can only be interesting in dimension 4. And indeed, we get that it is usually nontrivial, as was discussed after Proposition~\ref{prop:commutators}.
That result will follow from Proposition~\ref{prop:central} applied to $X=M_G$, see Section~\ref{sec:main-proofs}.

For the proof of Proposition~\ref{prop:central} we introduce a convenient notion of ``parametrized connected sums'' and prove a lemma about them, Lemma~\ref{lem:commutator}. 

Given $\ul{f_1},\ul{f_2}\in \pi_1(\Emb_\partial(\D^1,X),u)$, by a preliminary homotopy we may assume that they are represented by maps $f_i\colon\I\to\Emb_\partial(\D^1,X)$ supported on disjoint open subintervals $J_1,J_2\subseteq \D^1$, in the sense that for $\theta\notin J_i$ and all $t\in\I$ we have $f_i(t)(\theta)=u(\theta)$. By general position, we may also assume $f_i(\I)(J_i)\cap u(\D^1)=\emptyset$ for both $i$.

\begin{defn}\label{def:param-conn-sum}
    For $f_1,f_2\colon\I\to\Emb_\partial(\D^1,X)$ define the parametrized connected sum $f_1\#f_2\colon\I^2 \to \Imm_\partial(\D^1,X)$ by
\[
    f_1\#f_2(t_1,t_2)(\theta) \coloneqq
    \begin{cases}
        f_1(t_1)(\theta), & \text{ for } \theta\in J_1, \\
        f_2(t_2)(\theta), & \text{ for } \theta\in J_2, \\
        u(\theta), & \text{ for } \theta\in\D^1\sm(J_1\cup J_2).
    \end{cases}
\]
\end{defn}
The boundary $\partial(f_1\#f_2)\colon\partial\I^2\to \Emb_\partial(\D^1,X)$ is equal to the commutator $[f_1,f_2]=f_1 f_2 f_1^{-1} f_2^{-1}\in \Omega\Emb_\partial(\D^1,X)$. This can be seen by schematically labeling the domain square:
\begin{equation}\label{eq:commut-schema}\begin{tikzcd}
      u\arrow{r}[font=\scriptsize]{f_1} & u \\
      u\arrow{r}[swap,font=\scriptsize]{f_1} \arrow{u}{f_2}[swap,font=\scriptsize]{\quad f_1\#f_2} 
      & u \arrow{u}[swap,font=\scriptsize]{f_2}
\end{tikzcd}\quad\cong\quad
\begin{tikzcd}
     u\arrow[equals]{rr}\arrow[equals]{d}{\quad \quad f_1\#f_2} && u\arrow[equals]{d}\\
     u \arrow{rr}[swap,font=\scriptsize]{[f_1,f_2]} && u
\end{tikzcd}
\end{equation}
Moreover, each $f_1\#f_2(t_1,t_2)$ is a local embedding (hence an immersion), so $f_1\#f_2$ represents an element in $\pi_2(\Imm_\partial(\D^1,X),\Emb_\partial(\D^1,X),u)$ with boundary $\partial(\ul{f_1\#f_2})=[\ul{f_1},\ul{f_2}]\in\pi_1(\Emb_\partial(\D^1,X),u)$. We saw above that this relative group is isomorphic to $\Z[\pi]$ via $\Da$, computed using certain generic representatives in which only finitely many arcs have double points. We will now arrange this for our map $f_1\#f_2$.

First note that the arc $f_1\#f_2(t_1,t_2)$ is embedded if and only if we have $f_1(t_1)(J_1)\cap f_2(t_2)(J_2)=\emptyset$. This is generically not true for all $(t_1,t_2)\in\I^2$, but by general position $f_i$ can be chosen so that their adjoints $\I \times \D^1\to X$ are transverse, so intersect in isolated points $y_j=f_1(t_1^j,\theta_1^j)=f_2(t_2^j,\theta_2^j)\in X$ for $(t_i^j,\theta_i^j)\in\I \times \D^1$, $1\leq j\leq r$. As a consequence, the only arcs $f_1\#f_2(t_1,t_2)$ that are not embedded are those with $t_1=t_1^j$, $t_2=t_2^j$ for some $j$, and then $f_1\#f_2(t_1^j,t_2^j)(\theta_1^j)=f_1\#f_2(t_1^j,t_2^j)(\theta_2^j)=y_j$ the only double point of this arc.

\begin{lemma}\label{lem:commutator}
   If $J_1$ is before $J_2$ in the orientation of $\D^1$ then 
\[
    \Da(\ul{f_1\#f_2}) =\lambdabar(p_u\ul{f_1},p_u\ul{f_2}) \in\Z[\pi\sm1].
\]
\end{lemma}
\begin{proof}  
    Recall that $p_uf_i\colon\I^2\to \Omega X$ is the union of $(t\mapsto f_i(t)\cdot{u^{-1}})$ and canonical null homotopies of $u\cdot u^{-1}$ for $t\in\partial\I^2$ (recall that $\cdot$ concatenates $\D^1$'s). As $f_i(\I)(J_i)\cap u(\D^1)=\emptyset$, the transverse intersection points $\{y_1,\dots,y_r\}=p_uf_1\cap p_u f_2$, counted by the pairing $\lambda$, correspond exactly to the double points of those arcs $f_1\#f_2(t_1,t_2)$ that are not embedded, counted in the $\Da$ invariant. 
    The Dax loop $g_{y_j}$ goes along a whisker on $f_1$ from $x_-$ to $y_j$ and then back on $f_2$, exactly as in the formula computing $\lambda(p_u\ul{f_1},p_u\ul{f_2})$. 
    
    The signs also agree. If $\mathrm{sgn}_{y_j}(p_uf_1,p_uf_2)=+1$, then $df_1|_{(t^j_1,\theta_1^j)} \oplus df_2|_{(t^j_2,\theta_2^j)}$ orients $T_{y_j}X$, so with the standard vectors $(dt^j_1, dt^j_2)$ orients $T_{(t^j_1, t^j_2)}\I^2 \oplus T_{y_j}X$. For $\Da$, near $\theta^j_i$ the derivative $d(f_1\#f_2)$ is zero in the other direction $t_{3-i}$, so $d(f_1\#f_2)|_{(t_1^j,t_2^j,\theta_1^j)}$ is oriented as $- (dt^j_2\oplus df_1|_{(t_1^j,\theta_1^j)})$, as we had to flip the first two vectors, whereas $d(f_1\#f_2)|_{(t_1^j,t_2^j,\theta_2^j)})$ is oriented as $dt^j_1\oplus df_2|_{(t_2^j,\theta_2^j)}$. It follows that their sum
    orients $T_{(t^j_1, t^j_2)}\I^2 \oplus T_{y_j}X$ (since $dt^j_1$ passes 3 vectors to become first), so $\varepsilon_{y_j}=1$ by its definition~\eqref{eq:Dax-sign}.
\end{proof}

\begin{proof}[Proof of Proposition~\ref{prop:central}]
    The claimed commutator pairing follows immediately from Lemma~\ref{lem:commutator}, since $\partial\realmap$ is the inverse to $\Da$:
    \[
        [\ul{f_1},\ul{f_2} ] = \partial(\ul{f_1\#f_2})
        = (\partial\realmap\circ\Da)(\ul{f_1\#f_2})
        =\partial\realmap(\lambdabar(p_u\ul{f_1},p_u\ul{f_2})).
    \]
    As this clearly vanishes on $\ker(p_u)$, our extension is indeed central.
\end{proof}
Note that for the above proof we could have equally well chosen the opposite order $J_2<J_1$ in $\D^1$. Equivalently, keep $J_1<J_2$ but in Definition~\ref{def:param-conn-sum} use a map $f'_i$ supported on $J_{3-i}$, and isotopic to $f_i$. Observe that $\partial(f'_1{\#}f'_2)(t)=[f'_1,f'_2](t)$ is isotopic to $\partial(f_1\#f_2)(t)=[f_1,f_2](t)$ continuously in $t\in\partial \I^2$. Using this isotopy on an annulus extends $f'_1{\#}f'_2$ to a map
\[
    f_1\ol{\#}f_2\colon\I^2\to\Imm_\partial(\D^1,X),\quad\text{with}\quad\partial(f_1\ol{\#}f_2)=[f_1,f_2].
\]
In particular, $\partial(\ul{f_1\#f_2}-\ul{f_1\ol{\#}f_2})=0$ so the $\Da$ invariant of this difference has to be in $\ker(\partial\realmap)=\langle1,\md_u(\pi_3X)\rangle$. As in Lemma~\ref{lem:commutator} we find $\Da(f_1\ol{\#}f_2)=-\lambda(a_2,a_1)=-\ol{\lambda(a_1,a_2)}$, since now arcs $f_2$ appear as the $\theta_-$-sheet, so:
\[
    \lambda(a_1,a_2)-\ol{\lambda(a_1,a_2)}\in \md_u(\pi_3X).
\]
We now also identify the class in $\pi_3X$ witnessing this.
\begin{prop}\label{prop:Whitehead}
  For the Whitehead product $[a_1,a_2]_{\mathsf{Wh}}\in\pi_3X$ of two elements $a_i\in\pi_2X$ we have
   \[
        \md_u\big([a_1,a_2]_{\mathsf{Wh}}\big)=\lambdabar(a_1,a_2)+\lambdabar(a_2,a_1). 
   \]
\end{prop}
\begin{proof}
    Pick $f_i\colon\S^1\to\Emb_\partial(\D^1,X)$ so that $a_i= \ul{p_uf_i}$ and $f_1\#f_2$ generic as above. Then $f_1\ol{\#}f_2$ is also generic, and gluing them along $\partial\I^2$ gives a map $f_1\#f_2-f_1\ol{\#}f_2\colon\S^2\to\Imm_\partial(\D^1,X)$ for which we show 
    \[
        p_u(f_1\#f_2-f_1\ol{\#}f_2)=[a_1,a_2]_{\mathsf{Wh}}.
    \]
    Let $F_i\colon\S^1\to\Omega X$ be obtained from $p_uf_i$ by null homotoping its part which agrees with $u\cdot u^{-1}$, so $a_i=\ul{F_i}$. Similar homotopies, continuous in $(t_1,t_2)\in\I^2$ and agreeing on $\partial\I^2$, show that in $\Omega X$ we have
    \[
        (f_1\#f_2(t_1,t_2)) \cdot u^{-1}\simeq F_1(t_1)\cdot F_2(t_2)
        \ \text{ and }\  
        (f_1\ol{\#}f_2(t_1,t_2)) \cdot u^{-1}\simeq F_2(t_2)\cdot F_1(t_1)
    \]
   These homotopies glue to a homotopy rel.\ boundary from $p_u(f_1\#f_2-f_1\ol{\#}f_2)$, a map $\I^2\to\Omega X$ defined by $(t_1,t_2)\mapsto (f_1\#f_2-f_1\ol{\#}f_2)(t_1,t_2)$ union canonical null homotopies of $u\cdot u^{-1}\simeq c\coloneqq\const_{u(-1)}$ on $\partial\I^2$, to the map obtained by gluing squares $(t_1,t_2)\mapsto F_1(t_1)\cdot F_2(t_2)$ and $(t_1,t_2)\mapsto F_2(t_2)\cdot F_1(t_1)$. Schematically (on the left the canonical null homotopies are not drawn, but are used for the homotopy):
\begin{equation}\label{eq:htpy-to-samelson}
\begin{tikzcd}[column sep=large]
     uu^{-1}\arrow[equals]{rr}\arrow[equals]{d}{\quad\quad\quad (f_1\#f_2)\cdot u^{-1}} && uu^{-1}\arrow[equals]{d}\\
     uu^{-1} \arrow[dotted]{rr}[font = \scriptsize,swap,description]{[f_1,f_2]\cdot u^{-1}} && uu^{-1}\\
     uu^{-1}\arrow[equals]{rr}\arrow[equals]{u}[swap,near start]{\quad\quad -(f_1\ol{\#}f_2)\cdot u^{-1}} && uu^{-1}\arrow[equals]{u}
\end{tikzcd}\quad\simeq\quad\begin{tikzcd}
     c\arrow[equals]{rr}\arrow[equals]{d}{\quad \quad F_1\cdot F_2} && c\arrow[equals]{d}\\
     c \arrow[dotted]{rr}[font = \scriptsize,swap,description]{[F_1,F_2]} && c\\
     c\arrow[equals]{rr}\arrow[equals]{u}[swap]{\quad \; -(F_2\cdot F_1)} && c\arrow[equals]{u}
\end{tikzcd}
\end{equation}

    The Whitehead product $[a_1,a_2]_{\mathsf{Wh}}\in\pi_3X$ (in fact, its adjoint, the Samelson product, see \cite{GWhitehead-book}) is represented by the map $[F_1,F_2]\colon\I^2\to \Omega X$ taking $(t_1,t_2)$ to the commutator loop $[F_1(t_1),F_2(t_2)]\in\Omega X$, union canonical null homotopies $[F_1(t),u]\simeq c$ for $t\in\partial\I^2$. 
    By cutting the loop direction $\D^1$ in half, we can view the adjoint $[F_1,F_2]\colon\I^2 \times \D^1\to X$ as glued from two cubes along their faces $\I^2\times0$. After a manipulation which is the 3-dimensional analogue of~\eqref{eq:commut-schema}, these cubes become precisely the adjoints of the two squares on the right of~\eqref{eq:htpy-to-samelson}, showing that $[F_1,F_2]$ is homotopic to $p_u(f_1\#f_2-f_1\ol{\#}f_2)$.
\end{proof}


\subsection{Some algebraic topology of the dax homomorphisms}\label{sec:dax-proofs}

We next develop an important method to compute the invariant $\md$ from \eqref{eq:dax-def}, defined for a 4-manifold $X$ with basepoint in $\partial X$,  $\pi=\pi_1X$.

First recall that a \emph{quadratic} map $q\colon A\to B$ between abelian groups $A,B$ satisfies the defining condition that the map
\[
(a_1,a_2)\mapsto q(a_1+a_2) - q(a_1) - q(a_2) \in B
\]
is bilinear in $a_1,a_2\in A$. For example, the map $\pi_2X \to \pi_3X$ given by precomposition with the Hopf map $H\colon\S^3\to \S^2$ is well known to be quadratic, with associated bilinear pairing given by the Whitehead product $(a_1,a_2)\mapsto [a_1,a_2]_{\mathsf{Wh}}$.

Every abelian group $A$ admits a universal quadratic map $q_A\colon A\to \Gamma(A)$ with the property that precomposing with $q_A$ gives a bijection between quadratic maps $q\colon A\to B$ and homomorphisms $\Gamma(q)\colon\Gamma(A)\to B$. Here $\Gamma(A)$ is an abelian group depending functorially on $A$. For $A$ free abelian, $\Gamma(A)$ is simply the subgroup of symmetric tensors in $A\otimes A$, and $q_A(a) = a\otimes a$. However, one easily checks that $\Gamma(\Z/2) \cong \Z/4$ with $q_{\Z/2}(1\!\! \mod 2)= 1\!\! \mod 4$.

\begin{theorem}\label{thm:Hopf}
  There is a commutative diagram of short exact sequences of abelian groups
  \[\begin{tikzcd}
       \Gamma(\pi_2X)\arrow[tail]{rr}{\Gamma(-\circ H)}\arrow{d}{\Gamma(\mu_2)}
        && \pi_3X \arrow[two heads]{r}{\mathrm{Hur}} \arrow{d}{\md} 
        & H_3\wt X\arrow{d}{\mu_3} \\
        \faktor{\RGR}{ \langle \ol{g}-g \rangle}\arrow[tail]{rr}{g\ \mapsto\ g+\ol{g}} 
        && \RGR^\sigma \arrow[two heads]{r}{} 
        & \faktor{\Z[\pi]^\sigma}{\langle 1, g+\ol{g} \rangle}
    \end{tikzcd}
    \]
    In particular, $\md(a\circ H) = \mu_2(a) +\overline{\mu_2(a)}= \lambdabar(a,a)$ for all $a\in\pi_2X$. 
\end{theorem}
The horizontal maps on the right are the Hurewicz map for the universal covering $\wt X$ and the natural quotient map.
The bottom row is exact by construction, whereas the first row is Whitehead's ``certain exact sequence'' \cite{Whitehead-Certain-Exact-Seq}, with the homomorphism out of $\Gamma(\pi_2X)$ determined by the quadratic map that is precomposing with the Hopf map $H\colon\S^3\to \S^2$. It is injective as the next term in Whitehead's sequence is $H_4\wt X$ which vanishes in our setting (since $\partial X\neq\emptyset$).
Another homomorphism out of $\Gamma(\pi_2X)$ is determined by  Wall's reduced self-intersection invariant $\mu_2\colon\pi_2X\to\RGR/\langle\ol{g}-g\rangle$, analogous to the one from Proposition~\ref{prop-intro:4-term}.
Counting self-intersections of a connected sum, it follows that $\mu_2$ is quadratic:
\begin{equation}\label{eq:mu-2-quadratic}
    \mu_2(a_1+a_2)=\mu_2(a_1)+\mu_2(a_2)+[\lambdabar(a_1,a_2)].
\end{equation}
The last term is the class of the reduced intersection form $\lambdabar$, as in Proposition~\ref{prop:commutators} and \eqref{eq:lambdabar}, modulo $\langle\ol{g}-g \rangle$. Although $\lambda$ is hermitian, i.e.\ $\lambda(a_1,a_2)=\ol{\lambda(a_2,a_1)}$, this reduction makes it symmetric, as required by the $\Gamma$-functor.

Lastly, the map $\mu_3$ on the right is also a Wall invariant, counting self-intersections of $\pi_1$-trivial 3-manifolds immersed in $X \times \R^2$, where the $\R^2$-factor makes $\mu_3$ linear because the bilinear intersection term vanishes, see Section~\ref{subsec:Dax-FQ}. Its target ${\Z[\pi]^\sigma}/{\langle 1, g+\ol{g} \rangle}$ is isomorphic to $\FF_2[T_X]$ (recall that its cokernel is the target of the Freedman--Quinn invariant in Proposition~\ref{prop:disks-spheres}).

For closed 4-manifolds the entire diagram in Theorem~\ref{thm:Hopf} still exists, except for the map $\md$, whose definition requires a nonempty boundary. In fact, there is an $\S^2$-bundle $X$ over $\mathbb{RP}^2$ for which there cannot be a homomorphism $\pi_3X \to \RGR^\sigma$ making the left square commute, see Example~\ref{ex:closed-dax}. 



\begin{proof}[Proof of Theorem~\ref{thm:Hopf}]
    The upper sequence is exact by a classical result of Whitehead \cite{Whitehead-Certain-Exact-Seq}: the kernel of the Hurewicz homomorphism is equal to the image of $-\circ H\colon\Gamma(\pi_2X)\to\pi_3X$, where $\Gamma$ is a certain quadratic functor making this precomposition with the Hopf map $H\colon\S^3\to \S^2$ into a homomorphism. The lower sequence is exact since the kernel of the quotient map $q\colon \Z[\pi]^\sigma \sra \Z[\pi]^\sigma/\langle 1, g+\ol{g}\rangle$ is precisely the image of the norm map $\Z[\pi\sm1]\to\Z[\pi\sm1]^\sigma$, $g\mapsto g+\ol{g}$, whose kernel is generated by elements $\ol{g}-g$.
    
    In \cite[Lem.4.2]{ST-LBT} it was shown that $\mu_3$ factors through the Hurewicz map $\pi_3X \sra H_3(\wt X)$, and the right hand side square commutes by Lemma~\ref{lem:Dax-mu}.
    Moreover, since $\mu_3$ vanishes on $\im(-\circ H)$, on that image $\md$ must have values in $\ker(q)$. It remains to show that the left hand side square commutes, where both maps out of the $\Gamma$-group come from quadratic maps on $\pi_2X$. By the universal property of $\Gamma$ it suffices to show
    that $\md(a\circ H)$ equals $\mu_2(a) + \ol{\mu_2(a)}$ for all $a\in\pi_2X$. Counting intersection of $a$ with a push-off gives the well known formula  $\lambdabar(a,a)=\mu_2(a)+\ol{\mu_2(a)}$ that reduces things to $\lambdabar(a,a)$.
    
    Since $[\iota,\iota]_{\mathsf{Wh}}=2H$ we get $[a,a]_{\mathsf{Wh}}=a\circ[\iota,\iota]_{\mathsf{Wh}}=a\circ 2H=2(a\circ H)$. From Proposition~\ref{prop:Whitehead} it follows that $\md([a,a]_{\mathsf{Wh}}) = 2\lambdabar(a,a)$ and hence
    \[
        2 \cdot \md(a\circ H) = \md(2(a\circ H)) = \md([a,a]_{\mathsf{Wh}}) = 2\cdot \lambdabar(a,a).
    \]
    Notice that this equation lives in the abelian group $\Z[\pi\sm1]^\sigma$, which is torsion-free as a subgroup of the free abelian group $\Z[\pi\sm1]$. Therefore, we can divide both sides of the last equation by $2$, giving the desired result.
\end{proof}

For closed 4-manifolds, the outer homomorphisms $\mu_2,\mu_3$ in Theorem~\ref{thm:Hopf} still work but there cannot exist a homomorphism $\md$ making the left square commute, as the following example demonstrates.
\begin{example}\label{ex:closed-dax}
    Consider $\S^2 \times \S^2$ with generators $a_1, a_2 \in\pi_2(\S^2 \times \S^2)$ coming from the factors. Then the Whitehead product $[a_1,a_2]_{\mathsf{Wh}}$ vanishes by definition and so $a_1\otimes a_2 + a_2\otimes a_1\in \Gamma(\pi_2(\S^2 \times \S^2))$ is in the kernel of $-\circ H$. It is also the image of the generator coming from $H_4(\S^2 \times \S^2;\Z)$.
    
    The group $\Z/2= \langle t \rangle$ acts freely on $\S^2 \times \S^2$ via $t\cdot(v_1,v_2) = (-v_1, -v_2)$ and if we define $N$ to be its quotient then $\mu_2(a_1\otimes a_2 + a_2\otimes a_1) = \mu_2(a_1+a_2) - \mu_2(a_1)-\mu_2(a_2) = [\lambda(a_1,a_2)] = [1- t]=[-t] \neq 0$.
    Hence a homomorphism $\md$ making the left diagram in Theorem~\ref{thm:Hopf} commute cannot exist for $N$.

    However, if we remove an open $4$-ball from $N$ to create a $4$-manifold $X$ with boundary, then $-\circ H$ becomes an isomorphism and $\md([a_1,a_2]_{\mathsf{Wh}})$ becomes twice the generator in $\Z[\pi_1X\sm 1]^\sigma \cong\Z$ as required by our diagram. Since the projections of $a_i$ to $N$ are embedded spheres, $\mu_2$ vanishes on them, so does $\md$, and the target of the Dax invariant is $\Z[\pi_1X \sm 1]^\sigma/\md(\pi_3X)=\Z/2$ (and agrees with the target of the Freedman--Quinn invariant).
\end{example}
Recall that Lemma~\ref{lem:M-c} determines the image of the homomorphism $\md$ for the manifolds $M_c=\S^1 \times \D^3\cup_c h^2$.
\begin{proof}[Proof of Lemma~\ref{lem:M-c}]
    Whenever $\pi$ has no elements of order 2, the target of $\mu_3$ in Theorem~\ref{thm:Hopf} vanishes. All 4-manifolds $M_c$ have the homotopy type of $\S^1\vee \S^2$ because $c$ is null homotopic. It follows that $H_3\wt M_c=0$ and hence the image of $\md$ can be computed from $\Gamma(\mu_2)$. 

    For a free abelian group $A$, Whitehead's group $\Gamma(A)$ is freely generated by symmetric tensors $a_i\otimes a_i$ and $a_i\otimes a_j + a_j \otimes a_i$, where $a_i$ runs through an ordered basis for $A$ and $i>j$. Since $A=\pi_2M_c$ and $\pi_2M_c$ is a free $\Z[\pi]$-module with one generator $S_c$, our basis is $a_i=t^i\cdot S_c$, $i\in\Z$. Thus, 
    \begin{align*}
        \Gamma(\mu_2)(a_i\otimes a_i) &= \mu_2(a_i)  = t^i\cdot\mu_2(S_c)\cdot t^{-i}= \mu_2(S_c),\\
        \Gamma(\mu_2)(a_i\otimes a_j{+}a_j\otimes a_i) &= \mu_2(a_i{+}a_j) - \mu_2(a_i) - \mu_2(a_j) 
        \\&= \lambdabar(a_i,a_j)\, \wh{=}\, t^{i-j}\cdot\mu_2(S_c)
    \end{align*}
    in $t\cdot\Z[t]$.
    The last equality uses $\lambdabar(S_c,S_c)=\mu_2(S_c)+\ol{\mu_2(S_c)}$ and our identification of the targets of these maps with $t\cdot\Z[t]$. It follows that the image of $\Gamma(\mu_2)$ is exactly the ideal generated by $\mu_2(S_c)$.

    To realize a given polynomial $f$, we start with the unknot in $\S^1 \times \S^3$ and for each $\pm b_it^i$ in $f$ we do $b_i$ finger moves along the group element $t^i$, with the crossing change of correct sign $\pm$, see Figure~\ref{fig:Kirk}.
\end{proof}

Let now $M$ be a $4$-manifold and $k\colon\S^1\hra\partial M$ a knot null homotopic in $M$, and which has a dual $G\colon\S^2\hra \partial M$. As before, let $M_G$ denote the $4$-manifold obtained from $M$ by attaching a 3-handle along $G$, and $\pi\coloneqq \pi_1M$.
\begin{lemma}\label{lem:lambda-splits}
    The inclusion $M\subseteq M_G$ induces an isomorphism $\pi\cong\pi_1M_G$, a surjection $\pi_3M\sra\pi_3M_G$, and a short exact sequence of $\Z[\pi]$-modules
    \[\begin{tikzcd}
        \Z[\pi]\arrow[tail]{r}{\cdot G} 
        & \pi_2M\arrow[two heads]{r}
        & \pi_2M_G.
    \end{tikzcd}
    \]
    with a splitting on the left given by $\lambda^{rel}_M(\U,\bull)$ for the relative intersection form $\lambda^{rel}_M\colon\pi_2(M,\partial M)\times\pi_2(M)\to\Z[\pi]$ between disks and spheres.
\end{lemma}
\begin{proof}
    Since attaching a $3$-handle is homotopy equivalent to attaching a $3$-cell, we immediately get $\pi_iM\cong\pi_iM_G$ below degree $2$. Moreover, the relative homotopy group $\pi_3(M_G,M)$ is the free $\Z[\pi]$-module spanned by $h^3$. Once we show that homomorphism $\lambda^{rel}_M(\U,\bull)$ is a splitting, the surjectivity on $\pi_3$ will follow from the long exact sequence of a pair. Indeed, since $G$ is the geometric dual for $k=\partial\U$, we have $\lambda_{\partial M}(k,G)=1$, so a push-off of $G$ intersects $\U$ in the interior with $\lambda^{rel}_M(\U,G)=1$.
\end{proof}

Recall from \eqref{eq:dax-def} that a homomorphism $\md_X\colon\pi_3X\to\Z[\pi_1X\sm1]$ is defined for any 4-manifold $X$ with a choice of a basepoint in $\partial X$.
\begin{lemma}\label{lem:Dax-M-M_G}
    We have $\md_{M_G} (\pi_3M_G)=\md_M (\pi_3M)$ as subgroups of $\Z[\pi]$.
\end{lemma}
\begin{proof}
    The following diagram commutes:
    \[\begin{tikzcd}[column sep=large]
          \pi_3M\arrow[two heads]{r}{p}\arrow[bend right=15pt]{rr}[swap]{\md_{M}} & \pi_3M_G\arrow{r}{\md_{M_G}} & \Z[\pi]
        \end{tikzcd}
    \]
    since attaching a handle to $\partial M$ does not have influence on the calculation of $\md_{M}([f])\coloneqq\Da(F)$. Indeed, for $f\colon\S^3\to M$ represented by $F(\vec{t})\colon\D^1\imra M$, the same family also computes $\md_{M_G}([f])$, implying $\im\md_M\subseteq\im\md_{M_G}$. The other inclusion follows since the map $p$ is surjective by Lemma~\ref{lem:lambda-splits}: if $r=\md_{M_G}(a)$ for $a\in\pi_3M$, then $r=\md_M(b)$ for $b=p(a)$.
\end{proof}

For the following lemma we assume that $k\colon\S^1\hra M$ lies in a component $\partial_0M$ of $\partial M$ that is diffeomorphic to $\S^1 \times \S^2$. If $\U\colon\D^2\hra M$ has boundary $k$, we can use it to ambiently surger $\partial_0M$ into a 3-sphere $C\colon\S^3\hra M$ that decomposes $M$ into a connected sum $M\cong (\D^2 \times \S^2) \# M'$. Then $\U=\D^2 \times pt, \ k=\partial\D^2 \times pt, \ G = 0 \times \S^2$ and $\pi=\pi_1M\cong\pi_1M'$.
\begin{lemma}\label{lem:g+g-inv}
   In this setting, $\md(C)=0$ and for all $g\in\pi\sm1$ we have $\md(g\cdot[C])=g+\ol{g}$, where we use the usual $\pi$-action on $\pi_3M$.
\end{lemma}
\begin{proof}
    Recall that this action is by precomposing whiskers with loops, so the class $g\cdot[C]\in\pi_3M$ is represented by adding an embedded representative $\gamma\colon\S^1\hra M$ of $g$ to the whisker of $C$. Since our basepoint $x_-=u(-1)$ is in $\partial_0 M$, the circle $\gamma$ intersects $C$ transversely in two points of opposite sign as in Figure~\ref{fig:g+g-inv}.
    
    We now represent $g\cdot[C]$ by a 1-parameter family of arcs. For time $t\in[0,1/4]$ the arc $u$ (dashed in the figure) slides along $\gamma$ (take the midpoint of $u$ and drag it along $\gamma$) until the tip reaches $C$, then for $t\in[1/4,3/4]$ the arcs swing around $C$ (using our fixed foliation of $\S^2$ by a 1-parameter family of arcs), and finally for $t\in[1/4,3/4]$ return back to $u$ (run the family for $t\in[0,1/4]$ backwards). Since this family of arcs is obtained from the one for $[C]$ by conjugating all the arcs by $\gamma$, it represents $g\cdot[C]$.
    
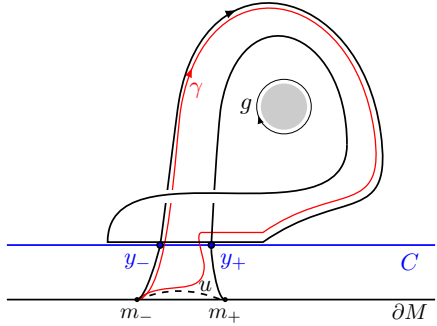
\begin{figure}[!htbp]
        \centering
        \begin{tikzpicture}[scale=0.68,every node/.style={scale=1},even odd rule]
    \clip (-1.1,-2.2) rectangle (9.1,5.6);
        \draw[blue,thick] (-1.1,-0.3) -- (9.1,-0.3);
        \draw[black,thick] (-1.1,-1.6) -- (9.1,-1.6);
        \draw[thick,dashed]
            (2, -1.6) to[out=20,in=160, distance=0.8cm] 
            (4, -1.6) node[above left]{$u$};
        \draw[semithick,red,postaction = decorate,decoration = {markings, mark = at position 0.68 with {\arrow{latex}} }]
                (2, -1.6) to[out=40,in=-90, distance=0.8cm] 
                (3.6, -0.8) to[out=90,in=-180, distance=0.2cm]
                (3.55, -0) -- (5,-0) to[out=35,in=-82, distance=2.1cm] 
                (7.65,2.4)
                (2, -1.6) to[out=5,in=-100, distance=0.8cm]
                (3.05, 3) node[above right]{$\gamma$} to[out=80,in=98, distance=3.55cm] (7.65, 2.4);
        \draw[thick,postaction = decorate,decoration = {markings, mark = at position 0.3 with {\arrow{latex}} }]
            (2.55, -0.3) to[out=80,in=-100, distance=0.4cm]
            (3,3.15) to[out=80,in=98, distance=3.6cm] (7.8,2.4)
            (3.77, -0.4) to[out=90,in=-100, distance=0.4cm]
            (4,3) to[out=80,in=90, distance=2.8cm] (7,2.1);
        \draw[,thick]
            (2,-1.6) to[out=0,in=-100, distance=0.2cm] (2.55, -0.3)
            (4.1,-1.6) to[out=180,in=-90, distance=0.2cm] (3.77, -0.4);
        \fill[white] 
            (2.3,0.7) rectangle (4,1);
        \draw[thick]
            (7,2.1) to[out=-90,in=90, distance=2.4cm] 
            (1.3,-0.23) --
            (5,-0.23) to[out=35,in=-82, distance=2.3cm] 
            (7.8,2.4);
        \fill[color=black!20!white] 
            (5.5,3) circle (0.55cm);
        \draw[postaction = decorate,decoration = {markings, mark = at position 0.55 with {\arrowreversed{latex}} }] 
            (5.5,3) circle (0.65cm) node[left=0.38cm]{$g$};
        
        \draw[thick]
            (2,-1.6) circle (1.1pt) node[below]{\small$x_-$}
            (4.1,-1.6) circle (1.1pt) node[below]{\small$x_+$}
            (8.5,-1.5) node[below]{\small$\partial M$};
        \fill[thick,blue,draw=black]
            (3.77, -0.3) node[below right]{$y_+$} circle (2pt)
            (2.55, -0.3) node[below left]{$y_-$} circle (2pt);
        \draw[blue] 
            (8.5,-0.3) node[below]{$C$};
\end{tikzpicture}
        \caption{The immersed arc in the family for $g\cdot[C]\in\pi_3M$.}
        \label{fig:g+g-inv}
\end{figure}
    There is a single arc which is not immersed, namely the one depicted in Figure~\ref{fig:g+g-inv}, with two double points, $y_-$ and $y_+$. We claim that the signed group element at $y_+$ is $+g$. Namely, the Dax loop is obtained by following the black arc from $x_-=u(-1)$ to the first occurrence of $y_+$, and then back on the black arc from the second occurrence of $y_+$ to $x_-$, is clearly homotopic to $\gamma$. To see the sign, note that the first sheet moves towards the reader and points right at $y_+$, whereas the second sheet points downwards at $y_+$, forming a positive basis of the present $\R^3$. Similarly, at $y_-$ the double point loop is given by $g^{-1}$, and we claim the sign is $-(-1)^3$. Namely, $(-1)^3$ arises since here the earlier time $\theta_-$ occurs on the sheet which is not moving, while the additional minus sign arises because that sheet has the opposite tangent vector (unchanged for the other sheet).
\end{proof}

This lemma was also was proven in \cite{K-Dax}. More generally, in that paper the first author studied a general 4-manifold $X$ and another useful tool for computing $\md$ -- namely, the action of $\pi\coloneqq\pi_1X$ on $\pi_3X$, and showed that if $\lambda(g\cdot a,g)\in\Z[\pi\sm1]$ denotes the equivariant intersection number between the 3-sphere $g\cdot a$ and the loop $g$, and $\lambdabar$ is $\lambda$ minus the coefficient at $1$, then the following formula holds in $\Z[\pi\sm1]$:
\begin{equation}\label{eq:dax-g}
    \md(g\cdot a) = g\cdot\md(a)\cdot\ol{g} - \lambdabar(g\cdot a,g) + \ol{\lambdabar(g\cdot a,g)},
\end{equation}
Observe that $\pi$ acts on all groups in Theorem~\ref{thm:Hopf}: on the bottom three by the conjugation action on $\Z[\pi]$, and on homotopy groups of $M$ by changing whiskers. The maps $\mu_2, \mu_3$ are $\pi$-equivariant, $\mu_i(g\cdot a) = g\cdot \mu_i(a) \cdot \ol{g}$, but for $\md$ we have \eqref{eq:dax-g}.
If $a=b\circ H$ for $b\in\pi_2M$ then $\lambdabar(g\cdot a,g)=0$, since $a(\S^3)=b(\S^2)$ is ``2-dimensional'', so generically disjoint from any 1-dimensional loop $g$. Thus, \eqref{eq:dax-g} is consistent with the equivariance of $\mu_2$.

\section{The relative Dax invariant for disks in 4-manifolds}\label{sec:disks}

In this section we discuss relative Dax invariants for both neat disks and half-disks in an arbitrary smooth, oriented 4-manifold $X$. We will later restrict to two settings: neat disks in $X=M$ with a dual for the boundary knot as in Section~\ref{sec:intro}, and half-disks in $X=M_G$ as in Section~\ref{sec:space-level}.

The first order of business is to translate the Dax invariant for families of arcs to an invariant of disks. To this end, we fix a convenient parametrization of $\D^2$ by the rectangle $\I \times \D^1=[0,1]\times[-1,1] \to \D^2$. This collapses the upper face $\I \times 1$ to the point $i\in\D^2$, the lower face $\I \times -1$ to the point $-i\in\D^2$, and restricts to a diffeomorphism $\I \times (-1,1) \cong \D^2 \sm \{\pm i\},(t,\theta)\mapsto a_t(\theta)$. 

We thus think of $\D^2$ as foliated by the smooth family of arcs $a_t\colon\D^1\hra \D^2$ each going from $-i$ to $i$, for $t\in[0,1]$. Then $a_0$ is the left semicircle, $a_1$ is the right semicircle and $a_{1/2}$ is the vertical arc through $0$. We can actually choose $a_t$ to be the projection to $\D^2\subseteq\R^2$ of the upper great circles in $\S^2\subseteq\R^3$. 

Then an embedding $K\colon\D^2\hra X$ gives a smooth family $K\circ a_t\colon\D^1\hra X$ of embedded arcs.
We now add the following requirements on $K$ that define our two central notions, neat embeddings versus embeddings of half-disks.
\begin{itemize}
\item[\emph{all}\;] 
    For all embeddings $K\colon\D^2\hra X$ we require that $K\circ a_t$ are neat for $t\neq 0,1$ and that $K\circ a_0\colon\D^1\hra \partial X$ lies in the boundary. 
\item[\emph{neat}\;] 
    We call $K$ \emph{neat} if $K\circ a_1\colon\D^1\hra\partial X$ again lies in the boundary and forms a smooth knot $K|_{\partial\D^2}\colon\S^1\hra\partial X$ together with $K\circ a_0$. 
\item[\emph{half}\;]
    We call $K$ a \emph{half-disk} if $K\circ a_1\colon\D^1\hra X$ is a \emph{neat} arc (that typically makes a ``right angle'' with $K\circ a_0$), called the ``free'' boundary of $K$.
\end{itemize}
In the neat case, this is just a convenient reformulation of the notion used in Section~\ref{sec:intro}.
Note that for a half-disk $K$, the domain of definition is still $\D^2$ and our half-disk symbol $\HD$ is just a reminder that $K$ has a free boundary arc (and is hence not neat). 
We now fix boundary conditions in both cases. 
\begin{defn}\label{def:boundary conditions}
    Fix a knot $k\colon\S^1\hra \partial X$ and let $\Emb(\D^2,X;k)$ be the space of neat embeddings with boundary $k$. We can decompose $k = k\circ a_0\cup k\circ a_1$ into two arcs that have the same endpoints $x_\pm\in\partial X$.  
    
    In addition, fix a neat arc $u\colon\D^1\hra X$ with endpoints $x_\pm$ and denote by $\Emb(\HD,X;k)$ the space of half-disks with boundary $k= k\circ a_0 \cup u$. We will sometimes also use the notation $k_-\coloneqq k\circ a_0$ so that $k=k_-\cup u$.
\end{defn}
Note that $\Emb(\HD,X;k)$ has a weaker boundary condition compared to $\Emb(\HD,X;k^\e)$ from Section~\ref{sec:space-level}. We will switch between them in Section~\ref{sec:main-proofs}.


\subsection{Dax invariant for half-disks and neat disks}\label{sec:Dax-half}
A homotopy between half-disks $K_0,K_1\in \Emb(\HD,X;k)$ can be parametrized by
\[
   H\colon \I \times (\I \times \D^1) \ra X, \quad (s,t, \theta) \mapsto H_s(t,\theta),
\]
with $H_0=K_0$ and $H_1=K_1$, such that each $H_s$ has the same boundary condition $k$. Then $(s,t) \mapsto H_s(t,-)$ defines a map $\I^2 \to \Map_\partial(\D^1,X)$ to which we would like to apply the $\Da$ invariant from Section~\ref{sec:Dax-arcs}. 

    

In Definition~\ref{def:Dax} the boundary condition $u$ occurs on all but one boundary face of $\I^2$ (on which it lies in embeddings), while here we have for $s\in \{0,1\}$ embeddings $K_0, K_1$ and for $t\in\{0,1\}$ the constant arcs $k_-=k\circ a_0$ respectively $u$. The former setup was convenient for the relative homotopy group $\pi_2(\Imm_\partial(\D^1,X),\Emb_\partial(\D^1,X),u)$ because a group structure on it is given by gluing maps of squares along one constant face. One can translate between a homotopy $H$ of half-disks and a representative $F$ of a relative homotopy class as in diagram \eqref{eq:square-convention}.
 \begin{equation}\label{eq:square-convention}
 \begin{tikzcd}
 \mbox{} & \\
  \arrow{u}{s} \arrow[swap]{r}{t} & \mbox{}
   \end{tikzcd}
   \quad 
     \begin{tikzcd}
        k_-\arrow{r}{K_1}\arrow[equals]{d}{\hspace{0.65cm} H} & u \arrow[equals]{d} \\
      k_-\arrow{r}[swap]{K_0} & u 
    \end{tikzcd}
    \quad\cong\quad
    \begin{tikzcd}
        u\arrow[equals]{r}\arrow[equals]{d}{\hspace{0.65cm} F} & u\arrow[equals]{d}\\
       u \arrow{r}[swap]{K_1^{-1}\cdot K_0} & u
    \end{tikzcd}
    \quad
    \begin{tikzcd}
 \mbox{} & \\
  \arrow{u}{t_2} \arrow[swap]{r}{t_1} & \mbox{}
   \end{tikzcd}
 \end{equation}
 On the left, each point $(s,t)$ corresponds to the immersed arc $H_s(t,-)$, whereas on the right $(t_1,t_2)$ corresponds to the immersed arc $F(t_1,t_2,-)$, and double lines denote constant faces. 
 More precisely, there is a diffeomorphism $\varphi$ of $\I^2$ with which we precompose the map $\I^2\to \Imm_\partial(\D^1,X)$ given by $H$ to get the one given by $F$. So $\varphi$ is the identity slightly away from the boundary, whereas on the boundary it fixes only the lower part of the right face and otherwise arranges exactly for the transition from $H$ to $F$ as in \eqref{eq:square-convention}. Here we use that the composition $K_1^{-1}\cdot K_0$ of the paths of arcs $K_1^{-1}, K_0\colon\I\to \Emb_\partial(\D^1,X)$ takes a short pause at $k_-$ by definition of their concatenation. Moving from boundary to the inside of the square, the diffeomorphism of the boundary is isotoped back to the identity to obtain $\varphi$. 

\begin{defn}
    \label{def:Dax(H)}
    For a homotopy $H$ of half-disks define $\Dax(H)\in \RGR$, with $\pi=\pi_1X$, as the sum over double points of a generic representative of
    \[
        \wt H\colon \I^2 \times \D^1 \to \I^2 \times X, \quad (s,t,\theta)\mapsto (s,t,H_s(t,\theta)).
    \]
    These double points occur at finitely many values of $(s,t)$, for which the arc $H_s(t)$ has a double point, with a well-defined signed fundamental group element. As before, we only count the nontrivial group elements.
\end{defn}
The following observation holds by comparing $\wt H$ and $\wt F$ via $\varphi$ as above.
\begin{lemma}\label{lem:Dax=Da} 
    Via the transition between $H$ and $F$ from \eqref{eq:square-convention}, the $\Dax$ invariant for homotopies of half-disks $\Dax(H)$ is equal to the $\Da$ invariant $\Da(F)$ for $[F]\in \pi_2(\Map_\partial(\D^1,X),\Emb_\partial(\D^1,X),u)$ from Section~\ref{sec:Dax-arcs}.
\end{lemma}
We can now prove one part of Lemma~\ref{lem:Dax-invt}, namely that $\Dax(H)$ is independent of the choice of a homotopy $H$ between half-disks $K_0,K_1$, modulo $\md(\pi_3M)$ as in Definition~\ref{def:dax}. In the next section we will show the remaining part, specific for \emph{neat disks}, see Remark~\ref{rem:Dax-neat}.

For homotopies $H$ from $K_0$ to $K_1$ and $H'$ from $K_1$ to $K_2$, let $H \cup_{K_1} H'$ be their concatenation along the top face $K_1$, as on the left of~\eqref{eq:squares-diagram}. This is a homotopy with $\Dax(H \cup_{K_1} H')= \Dax(H) + \Dax(H')$, since each point in the glued rectangle lies in exactly one of the two squares, so corresponds to an immersed arc either in $H$ or $H'$. This is analogous to the additivity of $\Da$, where two squares were instead glued along a vertical $u$ face.
\begin{equation}\label{eq:squares-diagram}
\begin{tikzcd}
    k_-\arrow{r}[font = \scriptsize]{K_2}\arrow[equals]{d}{\hspace{0.65cm} H'} & u \arrow[equals]{d}\\
    k_-\arrow{r}[swap,font = \scriptsize]{K_1}\arrow[equals]{d}{\hspace{0.65cm} H} & u\arrow[equals]{d} \\
    k_-\arrow{r}[swap,font = \scriptsize]{K_0} & u
\end{tikzcd}\hspace{2cm}
\begin{tikzcd}
    k_-\arrow{r}[font = \scriptsize]{K_0}\arrow[equals]{d}{\hspace{0.57cm} -H''} & u \arrow[equals]{d}\\
    k_-\arrow{r}[swap,font = \scriptsize]{K_1}\arrow[equals]{d}{\hspace{0.65cm} H} & u\arrow[equals]{d}\\
    k_-\arrow{r}[swap,font = \scriptsize]{K_0} & u
\end{tikzcd}\quad\sim\quad
\begin{tikzcd} 
    u\arrow[equals]{r}\arrow[equals]{d}{\hspace{0.65cm} H} & u\arrow[equals]{d}\\
    u\arrow{r}[swap,font = \scriptsize]{K_1^{-1}\cdot K_0}\arrow[equals]{d}{\hspace{0.55cm} -H''} & u\arrow[equals]{d}\\
    u \arrow[equals]{r} & u
\end{tikzcd}
\end{equation}
\begin{lemma} \label{lem:HDax}
    If $H,H''$ are two homotopies from $K_0$ to $K_1$, then we have $\Dax(H) - \Dax(H'')\in \md(\pi_3X)$. In particular, the \emph{relative} $\Dax$ invariant $\Dax(K_0,K_1)\coloneqq [\Dax(H)]\in \RGR/\md(\pi_3X)$ is a well-defined isotopy invariant of half-disks. By definition, it satisfies $\Dax(K_1,K_0)=-\Dax(K_0,K_1)$
\end{lemma}
\begin{proof}
    We can view $-H''$ as running backwards from $K_1$ to $K_2=K_0$, so precisely one time direction is reversed, cf.\ \eqref{eq:square-convention}. This implies $\Dax(-H'')=-\Dax(H'')$. Then by the last paragraph we have $\Dax(H) - \Dax(H'')=$
    \[
     = \Dax(H)+\Dax(-H'')=\Dax(H \cup_{K_1} -H'')=\Dax(H\cup_{K_1^{-1}\cdot K_0}-H'').
    \]
    The last equality follows by \eqref{eq:square-convention}, see the right hand side of~\eqref{eq:squares-diagram}. The final Dax invariant lies in $\md(\pi_3X)$ by definition, since this family is now equal to $u$ along all of the boundary $\partial \I^2$. Finally, if $H$ is an isotopy we clearly have $\Dax(H)=0$, so this is an obstruction to isotopy.
\end{proof}

\begin{remark}\label{rem:Dax-neat} 
    The entire discussion in this section can be repeated word by word for neat disks with a boundary condition $k\colon\S^1\hra\partial X$. So a homotopy $H$ (rel.\ boundary) has an invariant $\Dax(H)\in \RGR$, that modulo $\md(\pi_3X)$ depends only on the endpoint neat disks, giving $\Dax(K_0,K_1)$. In fact, in the case of neat disks the image of $\Dax$ lies in the subgroup $\RGR^\sigma<\RGR$, as we will prove in Corollary~\ref{cor:Dax-FQ} below, completing the proof of Lemma~\ref{lem:Dax-invt}.
\end{remark}

\subsection{Dax invariant and Wall's self-intersection invariant}\label{subsec:Dax-FQ}

We would like to relate the invariant $\Dax(H)$ from the previous section to Wall's self-intersection invariant $\mu_3(B)$ for a generic map $B\colon\D^3\imra P$, where $P$ is any smooth 6-manifold and we assume that $B$ takes a basepoint in $\S^2$ to the basepoint in $P$. Being generic means that the only singularities of $B$ are isolated transverse double points in the interior of the 3-ball. In particular, the restriction $B|_{\partial\D^3}$ is an embedding that does not meet the interior of $B$.
\begin{defn}\label{def:mu_3}
 The self-intersection invariant $\mu_3(B)\in \Z[\pi_1P]/\langle 1, g+\ol{g}\rangle$ is  given by the sum of signed group elements $g_p\coloneqq B(w_a)B(w_b)^{-1}$ over all double points $y=B(a)=B(b)$ of $B$, where $w_a,w_b\colon\I\to\D^3$ are arbitrary whiskers from a basepoint in $\partial\D^3$ to $a,b\in\D^3$. There is no preferred order of the sheets $a,b$ so we have to mod out $g+\ol{g}$ to remove this ambiguity.
\end{defn}
Then $\mu_3(B)$ is invariant under homotopies $B_s$ of the 3-ball which satisfy $B_s(\partial\D^3)\cap B_s(\D^3\sm\partial\D^3)=\emptyset$ for all $s\in\I$. This condition avoids moving a double point off across the boundary, which would change $\mu_3(B)$.
\begin{remark}
    \label{rem:generic}
    To show invariance of $\mu_3(B)$ under such homotopies, one perturbs them to become generic which reduces to the following kind of homotopies (that in particular restrict to isotopies on $\partial\D^3$): 
\begin{itemize}
\item 
    1-parameter families of generic maps -- these are given by pre- and post-composition with ambient isotopies,
\item 
    1-parameter families of maps that are either generic or have interior self-tangencies of codimension 1 -- these are given by finger moves and Whitney moves in the interior of $B$.
\item 
    cusp homotopies that introduce a single interior self-intersection with trivial group element.
\end{itemize}
\end{remark}
 Note  that $\mu_3$ is in particular preserved by homotopies through families of neat maps $B_s\colon(\D^3,\S^2) \to (P,\partial P),\ s\in\I$, that restrict to an isotopy of embeddings on the boundary. 
 
\begin{lemma}\label{lem:Dax-mu}
    For a generic representative $F\colon\I^2\to \Map_\partial(\D^1,X)$ of a class in $\pi_2(\Map,\Emb)$ as in Section~\ref{sec:Dax-arcs}, the induced map $\wt F\colon\I^2 \times \D^1 \imra \I^2 \times X$ is generic in the sense above Definition~\ref{def:mu_3} and $\mu_3(\wt F)=q(\Da(F))$, where 
    \[
    q\colon\RGR\sra \RGR/ \langle g+\bar g \rangle = \Z[\pi]/\langle 1, g+\ol{g}\rangle
    \]
    is the projection and $\pi=\pi_1X$.
    
    As a consequence, we have $\mu_3(\wt H) = q(\Dax(H))$ for a homotopy $H$ (rel.\ boundary) between neat or half-disks as in the previous section, and $q\circ \md$ is equal to $\mu_3$ applied to (homotopy classes of) maps $\S^3\to X$.
\end{lemma}
\begin{proof}
    Both invariants, $\mu_3(\wt F)$ and $\Da(F)$ count double points of $\wt F$ with groups elements and signs. The only difference is the particular choice of sheets in the definition of $\Da(F)$, where the order of inverse images in $\D^1$ is used to distinguish a double point loop $g$ from $- \bar g$. To prove the property $\mu_3(\wt F)=q(\Da(F))$ it thus suffices to compare the definitions of double points and signs on both sides.
      
    For $\Da(F)$ the fundamental group element $g_y$ at a double point $y=(\vec{t},F(a))=(\vec{t},F(b))\in\I^2 \times X$ with $\vec{t}=(t_1,t_2)$ and $a=(\vec{t},\theta_-)$, $b=(\vec{t},\theta_+)$ is defined as the concatenation $g_y=w_aw_b^{-1}$ for $w_a=F(\vec{t},-)|_{[-1,\theta_-]}$ and $w_b=F(\vec{t},-)|_{[-1,\theta_+]}$, see~\eqref{eq:Dax-dp-loop}. 
       
    Observe that these two arcs are a particular choice for whiskers used for computing the double point loop for $\mu_3$. Moreover, in both cases the signs are computed as in \eqref{eq:Dax-sign}. 
\end{proof}

\subsection{The Freedman--Quinn invariant for neat disks} \label{sec:FQ}
 
We now turn to the definition of the Freedman--Quinn invariant that uses a less clever way to turn a homotopy $H$ between disks into a generic map $\wh{H}$ of a 3-manifold into a 6-manifold. The main advantage of this second method is that it agrees with the above discussion in the neat case, and can be used to show that $\Dax$ gives elements in the fixed point set $\RGR^\sigma$, using the following result.
\begin{lemma} \label{lem:mu_3}
    If $P^6=\I \times Q^5$ and $B\colon \D^3\imra P$ has $B(\S^2) \subseteq \{\frac{1}{2}\} \times \partial Q \subseteq \partial P$, then $\ol{\mu_3(B)} = \mu_3(B)$. 
\end{lemma}
\begin{proof}
    This is a consequence of the formula $\mu_3(B)-\ol{\mu_3(B)} = \lambda_P(B,B')$ for a 3-ball in any 6-manifold $P$ and $B'$ any push-off of $B$. Under our assumption, we can find an extension $\beta\colon \D^3 \imra \frac{1}{2} \times Q$ of $\partial B\colon\S^2\hra \frac{1}{2} \times \partial Q$ that is homotopic (rel.\ boundary) to $B$. 
    Therefore, we only need to prove the property for $\beta$ in place of $B$. However, we can arrange that $\beta$ and its push-off $\beta'$ have distinct $\I$-coordinates and hence are disjoint, so their intersection number $\lambda_P(\beta,\beta')=\lambda_P(B,B')$ vanishes.
\end{proof}
Consider a homotopy $H\colon\I \times \D^2 \to X$ between $K_i\colon\D^2\hra X$. As in the previous section we use the coordinates $\I \times \D^1\sra\D^2$, but now study the map
\[
    \wh{H}\colon \I \times (\I \times \D^1) \to \I^2 \times X,\quad
    (s,t,\theta)\mapsto (\frac{1}{2},s,H_s(t,\theta))
\]
where the pair $(\I,\frac{1}{2})$ plays the same exact role as $(\R,0)$ in \cite{ST-LBT}.
One easily checks that $\wh{H}$ satisfies the assumption of Lemma~\ref{lem:mu_3} if and only if $K_i$ are neat embeddings -- if they are half-disks then the boundary points $(s,1,\theta)$ map to $(\frac{1}{2},s,H_s(1,\theta)) = (\frac{1}{2}, s, u(\theta))$ that do not all lie in $\{\frac{1}{2}\} \times \partial(\I \times X)$, because $u\colon\D^1\hra X$ is a neat arc. 

However, for neat disks $K_i$ we deduce from Lemma~\ref{lem:mu_3} that the element $\mu_3(\wh{H})$ is fixed under the involution. Therefore, we get the following straightforward generalization to neat disks of the Freedman--Quinn invariant for spheres from \cite{ST-LBT}. 
\begin{defn} \label{def:FQ}
    The relative Freedman--Quinn invariant of homotopic neat disks $K_i\in\Emb(\D^2,X;k)$ is defined by
    \[
        \FQ(K_0,K_1)\coloneqq [\mu_3(\wh{H})]\;\in\;\faktor{\Z[\pi]^\sigma}{\langle 1, g+\ol{g},\mu_3(\pi_3X)\rangle}\;\cong\;\faktor{\FF_2[T_X]}{\mu_3(\pi_3X)}.
    \]
\end{defn}
Recall that $\pi\coloneqq\pi_1X$, and $\Z[\pi]^\sigma$ denotes the fixed point set of $\sigma(g)=\ol{g}$. The isomorphism used in the definition comes from $\FF_2[T_X] \cong \Z[\pi]^\sigma/\langle 1, g+\ol{g}\rangle$, induced by the inclusion $T_X\subseteq\pi$ of the set of nontrivial 2-torsion elements. 
\begin{cor}
\label{cor:Dax-FQ}
    For neat disks $K_i$, $\FQ(K_0,K_1)=q(\Dax(K_0,K_1))$ and 
$\Dax(K_0,K_1)$ takes values in the fixed set of the involution $\sigma$. In particular, for any 4-manifold $X$ we have a homomorphism $\md\colon\pi_3X\to\RGR^\sigma$.
\end{cor}
\begin{proof}
    By Lemma~\ref{lem:Dax-mu} we have $\mu_3(\wt{H})=q(\Dax(H))$, so it suffices to show that $\mu_3(\wh{H})=\mu_3(\wt{H})$. Recall that $\wt{H}(s,t,\theta)=(s,t, H_s(t,\theta))$ was used to define $\Dax(H)$, whereas $\wh{H}$ is used to define $\FQ$. Both $\wt{H}$ and $\wh{H}$ are immersions $(\D^3, \S^2) \imra (P,\partial P)$ where $\D^3 \cong \I^2 \times \D^1$ and $P=\I^2 \times X$, so it suffices to find a homotopy between them that restricts to an isotopy on the boundary.
    Observe that we can certainly place the constant $\frac{1}{2}$ also in the second component of $\wh{H}$ without changing $\mu_3$. Then we get the required homotopy
\[
    \I \times \I^2 \times \D^1 \ni (r,s,t,\theta)\mapsto (s,r/2 +t(1-r),H_s(t,\theta)) \in\I^2 \times X
\]
    and one easily checks that for each $r\in\I$ it to an embedding on the boundary.
        
    Finally, as $\ker(q)$ by definition consists of elements fixed under the involution $\sigma$, and $q(\Dax(H))=\mu_3(\wh{H})$ is fixed by Lemma~\ref{lem:mu_3}, so is $\Dax(H)$.
\end{proof}

\subsection{Deducing previous results for spheres}\label{sec:spheres-proofs}

There is one assumption under which our results are equivalent to previous results for spheres. Namely assume that $\partial M$ has a connected component $\partial_0M$ diffeomorphic to $\S^1 \times \S^2$, and $k$ and $G$ are duals in $\partial_0M$ corresponding to $\S^1 \times p$ and $q \times \S^2$.

Given a neat disk $\U\colon\D^2\hra M$ with boundary $k$, the union of the ambient 2-handle $\nu\U$ and a collar in $M$ of $\partial_0M$ leads to a connected sum decomposition $M\cong (\D^2 \times \S^2) \# M'$ along the separating sphere $C\colon\S^3\hra M$. We can now apply Lemma~\ref{lem:g+g-inv} which is precisely about this setting: it says that $\md(g\cdot C)=g+\ol{g} \in \Z[\pi \sm 1]$ for every $g\in \pi\sm1$. 
\begin{cor}\label{cor:S1xS2}
    In this setting we have
   \[
        \faktor{\Z[\pi]^\sigma}{\langle 1,\md(\pi_3M) \rangle} = \faktor{\Z[\pi]^\sigma}{\langle 1, g+\ol{g}, \md(\pi_3M) \rangle} 
        \cong \faktor{\mathbb{F}_2 T_M}{\mu_3(\pi_3M)} 
    \] 
    So in this case the relative Dax and Freedman--Quinn invariants for neat disks take values in the same group.
\end{cor}
The last equality was explained in Definition~\ref{def:FQ} of the Freedman--Quinn invariant $\FQ$: recall that its target is precisely this quotient of the $\FF_2$-vector space generated by $T_M\coloneqq\{g\in\pi\sm1 \mid g^2=1\}$, by the image of Wall's (reduced) self-intersection invariant $\mu_3\colon\pi_3M\to\Z[\pi]/\langle1,g+\ol{g}\rangle$.


\begin{proof}[Proof of Proposition~\ref{prop:disks-spheres}]
    We will show that all vertical maps are bijections and that the first square with $\Dax$ and $\FQ$ maps commutes, while the rest of the diagram commutes by construction. 
    
    Observe that $N$ is obtained from $M\coloneqq N\sm\nu G$ by attaching a 2-handle with core $m_G$ to $k$ and a 4-handle. Since by assumption there exists a dual sphere $S$ for $G$, the circle $k$ bounds the disk $\U\coloneqq S\sm m_G$ in $M$, so it is null homotopic in $M$. Thus, the inclusion $i\colon M \hra N$ induces a canonical isomorphism $\pi=\pi_1M\cong\pi_1N$, so the rightmost vertical map is an isomorphism. 
    
    The Hurewicz theorem $\pi_2M \cong H_2(\wt M)$ and $\pi_2N \cong H_2(\wt N)$, the fact that no 3-handles are attached and $\lambda(S,G)=1$, imply a short exact sequence 
    \[\begin{tikzcd}
        \pi_2M \rar[tail]{i_*} &
        \pi_2N \rar[two heads]{\lambda(\bull,G)} &  \Z[\pi].
    \end{tikzcd}
    \]
    From the bijection $[\D^2,M;k] \cong \pi_2M$, $J\mapsto -\U\cup J$ from~\eqref{eq:cup-U} it follows that the second to right map is a bijection as well:
    \[
        [\D^2,M;k] \to [\S^2,N]^G,\quad J\mapsto J\cup m_G = (-\U\cup J) \# (\U\cup m_G).
    \]
    The map $q$ from Corollary~\ref{cor:Dax-FQ} takes the relative $\Dax$ invariant to the Freedman--Quinn invariant for neat disks in $M$, now becomes an isomorphism thanks to Corollary~\ref{cor:S1xS2}:
    \[
        \Dax(K_0,K_1)\in \faktor{\Z[\pi]^\sigma}{\langle 1,\md(\pi_3M) \rangle} 
        \cong \faktor{\mathbb{F}_2 T_M}{\mu_3(\pi_3M)} \ni \FQ(K_0,K_1).
    \] 
    By construction $\FQ(K_0,K_1)$ maps under $i$ to the invariant $\FQ(F_0,F_1)$ for spheres $F_i=K_i\cup m_G$ in $N$. Thus, once we show $i_*(\mu_3(\pi_3M))= \mu_3(\pi_3N)$, the leftmost vertical map will be an isomorphism and the first square will commute. 
    This equality of subsets of $\mathbb{F}_2T_N$ is similar to Lemma~\ref{lem:Dax-M-M_G}, except that $\mu_3$ factors through the (surjective) Hurewicz maps, so it suffices to see that $i_*\colon H_3(\wt{M})\to H_3(\wt{N})$ is surjective. Indeed, the 2-handle is attached in a homotopically trivial way so leaves $H_3(\wt{M})$ unchanged.
    
    We are left with the most interesting vertical arrow, second from the left, $\Dk \to \SG$, $K\mapsto K\cup m_G$. 
    It is surjective because for any $F\in \SG$ we can arrange that it is isotopic to $m_G$ near $G$. Namely, use local coordinates in which $G$ and $m_G$ are linear around $F\cap G=m_G\cap G$, so the inverse function theorem applied to the restriction of $F$ to an $\R^2$ neighborhood of the point $F^{-1}(F\cap G)$, that is $F_|\colon\R^2 \hra \nu(F\cap G) = \R^2 \times \R^2 \to\R^2$, implies that $F$ is locally a graph over $m_G$ and hence can locally be isotoped to it; this isotopy can be extended to all of $F$, keeping $G$ fixed. 
    
    Finally, the injectivity of this map follows from the commutativity of the square with the $\Dax$ and $\FQ$ invariants, and the exactness of the top sequence: If $K_0$ and $K_1$ lead to isotopic spheres $F_i$ then $\Dax(K_0,K_1)=\FQ(F_0,F_1)=0$, hence $K_0,K_1$ are also isotopic.
\end{proof}

In this proof, we started with the 4-manifold $N$ and removed $\nu G$ to create $M$ with a new boundary component diffeomorphic to $\S^1 \times \S^2$. Conversely, we may start with a 4-manifold $M$ with such a boundary component (cf.\ Example~\ref{ex:interior-conn-sum}) and add $\D^2 \times \S^2$ to $M$ along it. If there exists $\U\colon\D^2\hra M$ with boundary $k=\S^1 \times pt$, then this larger 4-manifold $N$ contains in its interior a framed sphere $G=0 \times \S^2$, dual to a sphere $F_\U\coloneqq\U\cup_s(\D^2 \times pt)$. The normal Euler number of $F_\U$ depends on the precise way we glue $\D^2 \times \S^2$ to $M$, our $\Z$ choices being parametrized by $\pi_1(SO_2)\sra \pi_1(SO_3) \leq \Diff(\S^1 \times \S^2)$, if we want our disks to match up along $k$. Due to the factorization over $\pi_1(SO_3)=\Z/2$, there are at most two diffeomorphism types of manifolds $N$ that can arise, differing by a Gluck twist. Then Proposition~\ref{prop:disks-spheres} shows that the previous LBT for spheres in $N$ implies our LBT for disks in $M$.

\section{Remaining proofs}\label{sec:proofs}
Recall from the introduction our geometric action of $\Z[\pi]$ on the set of isotopy classes $\Dk$ of neat disks in $M$ with boundary $k$ for which there is a dual $G\colon\S^2\hra\partial M$. For example, $1\in\pi$ acts by sending $K$ to $K^G_{tw}$, obtained by an interior twist on $K$ as in Figure~\ref{fig:interior-twist} and tubing into $G$. When we turn these neat disks in $M$ into half-disks in $M_G = M\cup_G h^3$ as in Theorem~\ref{thm:space-LBT}, the resulting half-disks $K'$ and $K'_{tw}$ become isotopic as will see in Corollary~\ref{cor:U_tw}.

However, the remaining $\RGR$-action on $\Dk$ does descend to the set $\HD(M_G;k)$. We now define this action for half-disks in an arbitrary 4-manifold $X$, with boundary $k=k_-\cup u$ as before (see the introduction to Section~\ref{sec:disks}).

\subsection{Geometric actions on half-disks in 4-manifolds}\label{sec:actions}
We now define a $\RGR$-action on isotopy classes of half-disks $\HD(X;k)$, with $\pi=\pi_1X$.
\begin{defn}\label{def:half-action}
    The action $(J,r)\mapsto J + \fm(r)$ of $\RGR$ on $\HD(X;k)$ is given by writing $r=\sum_i\e_i g_i$, with $\e_i\in\{\pm 1\}$, $g_i\in\pi\sm1$, and performing the following maneuvers for each signed group element $\e_i g_i$.
\begin{figure}[!htbp]
    \centering
    \begin{minipage}[t]{.64\textwidth}
        \centering
        \begin{tikzpicture}[scale=0.65,every node/.style={scale=1.3},even odd rule]
    \clip (-1.5,-2.8) rectangle (13.45,5.8);
        \draw[thick, blue!80] 
             (-1,-2) to[out=70,in=180,distance=3cm] (6, 1.2)  to[out=0,in=110,distance=3.1cm] (13,-2);
         \draw[thick, red!80] 
             (13,-2) -- (-1,-2);
        \fill[white] 
            (2.4,0.6) rectangle (4,1.3) 
            (7, 0.9) rectangle (7.4,1.4)
            (8.2, 0.9) rectangle (8.5,1.4);
        \fill[draw=black,line width=0.5pt,fill=yellow!70!red!20!white, fill opacity=0.3]
                        (1.75,-1.5) to[out=20,in=-100, distance=0.6cm]
                        (3,3.15) to[out=80,in=98, distance=4cm] (8.3,2)
                         to[out=-150,in=-20, distance=0.4cm]
                        (7.25,1.8) to[in=80,out=85, distance=3cm] (4,3) 
                         to[in=180,out=-100, distance=0.5cm] (4.05,-1);
        \draw (3.78,0.5) to[out=-100,in=-60, distance=0.4cm] (2.45,0.25);
        \draw[dotted]
            (3.78,0.5) to[out=150,in=80, distance=0.4cm] (2.45,0.25)
            (8.3,2) to[out=150,in=20, distance=0.4cm] (7.25,1.8);
        \draw[dotted,line width=0.8pt]
                (7.4, -0.4) to[out=-70, in=-100, distance=0.8cm] (8.3, 0.4)  ;
        \draw[line width=0.6pt]
                (7.25,1.8) to[out=-95, in=110, distance=0.5cm] (7.4, -0.4)
                (8.3, 0.4) to[out=80, in=-80, distance=0.4cm] (8.3,2);
        \fill[color=black!30!white] (5.6,3) circle (0.55cm);
        \draw[postaction = decorate,decoration = {markings, 
        mark = at position 0.55 with {\arrowreversed{latex}} }] 
            (5.6,3) circle (0.65cm) node[left=0.25cm]{$g$};
        \draw[thick]
            (-1,-2) circle (1.1pt) node[below]{$x_-$}  
            (13,-2) circle (1.1pt) node[below]{$x_+$};
        \draw 
            (9.8,-2) node[below,red!80]{$u$}
            (11.1,0.8) node[blue!80]{$k_-$};
        \fill[thick, fill=yellow!70!red!20!white, fill opacity=0.3] 
             (13,-2) -- (-1,-2) to[out=70,in=180,distance=3cm] (6, 1.2)  to[out=0,in=110,distance=3.1cm] (13,-2);
        \draw[thick,black, opacity=0.8]
            (7.42, -0.44) circle (1.8pt) (7.25, -0.85) node{$y_+$} 
            (8.3, 0.39) circle (1.8pt) (8.9, 0.35)  node{$y_-$};
\end{tikzpicture}
        \caption{The part of a finger move in present; two open disks from the tip of the finger are in past and future.}
        \label{fig:fm}
    \end{minipage}%
    \begin{minipage}[t]{.36\textwidth}
        \centering
        \begin{tabular}{@{}c@{}}
\begin{tikzpicture}[scale=0.65,every node/.style={scale=1.3},even odd rule]
    \clip (3,-2.6) rectangle (11.45,2);
        \draw[thick, blue!80] 
             (-1,-2) to[out=70,in=180,distance=3cm] (6, 1.2)  to[out=0,in=110,distance=3.1cm] (13,-2);
        \fill[white] 
            (7, 0.9) rectangle (7.4,1.4);
        \draw[dotted,line width=0.7pt]
            (7.4,-0.4) to[out=-70, in=140, distance=0.5cm] (8.2, -1.7);
        \fill[thick, fill=yellow!70!red!20!white, fill opacity=0.3] 
             (13,-2) -- (-1,-2) to[out=70,in=180,distance=3cm] (6, 1.2)  to[out=0,in=110,distance=3.1cm] (13,-2);
        \draw[thick, red!80] 
             (13,-2) -- (-1,-2);
        \draw[line width=0.7pt, postaction = decorate,decoration = {markings, mark = at position 0.48 with {\arrowreversed{latex}} }]
                (7.25,1.8) to[out=-100, in=110, distance=0.6cm] (7.4, -0.4);
        \draw[thick]
            (-1,-2) circle (1.1pt) node[below]{\small$x_-$}  
            (13,-2) circle (1.1pt) node[below]{\small$x_+$};
        \draw 
            (9.8,-2) node[below,red!80]{$u$}
            (11.1,0.8) node[blue!80]{$k_-$};
        \draw[thick,black, opacity=0.8]
            (7.4, -0.4) circle (1.8pt) 
            (8, -0.55) node{$y_+$};
        \draw[thick, dashed,orange] 
            (7.35, -0.45) -- (5.6, -2);
        \fill[orange] (5.6, -2) circle (2pt);
\end{tikzpicture}
\\
\begin{tikzpicture}[scale=0.65,every node/.style={scale=1.3},even odd rule]
    \clip (3,-2.6) rectangle (11.45,2);
        \draw[thick, blue!80] 
             (-1,-2) to[out=70,in=180,distance=3cm] (6, 1.2)  to[out=0,in=110,distance=3.1cm] (13,-2);
        \fill[white] 
            (7, 0.9) rectangle (7.4,1.4);
        \draw[thick, red!80] 
             (13,-2) -- (-1,-2);
        \fill[white]
            (4.5, -2.1) rectangle (4.8,-1.8);
        \draw[dotted,line width=0.7pt]
            (6.8,-1.9) to[out=35,in=110,distance=0.7cm] (8.2, -1.7);

        \fill[thick, fill=yellow!70!red!20!white, fill opacity=0.3] 
             (13,-2) -- (-1,-2) to[out=70,in=180,distance=3cm] (6, 1.2)  to[out=0,in=110,distance=3.1cm] (13,-2);

        \draw[thick]
            (-1,-2) circle (1.1pt) node[below]{\small$x_-$}  
            (13,-2) circle (1.1pt) node[below]{\small$x_+$};
        \draw 
            (9.8,-2) node[below,red!80]{$u$}
            (11.1,0.8) node[blue!80]{$k_-$};
        \draw[thick,orange, opacity=0.8]
            (7.4, -0.4) circle (1.8pt) 
            (8, -0.55) node{$b_+$};
        \draw[thick, dashed,orange] 
            (7.35, -0.45) -- (5.6, -2);
        \fill[orange] (5.6, -2) circle (2pt);
        \draw[line width=0.7pt, postaction = decorate,decoration = {markings, mark = at position 0.35 with {\arrowreversed{latex}} }]
            (7.25,1.8) to[out=-85,in=30,distance=1.2cm] 
            (5.5, -1)  to[out=-150,in=-145,distance=2cm] 
            (6.6,-2.1);
\end{tikzpicture}
\end{tabular}
        \caption{Push the double point $y_+$ along dashed arc across $u$.}
        \label{fig:push-across}
    \end{minipage}
\end{figure}

    To define the half-disk $J+\fm(g)$ for $g\neq1$ first do a finger move along $g$ on $J$, creating a generic immersion $J_g$ that has a single pair of transverse double points $y_\pm=J_g(a_\pm)= J_g(b_\pm)$ with opposite signs (let the $a$-sheet be on the finger), see Figure~\ref{fig:fm}. Then push off the points $a_+$ and $b_-$ across the free boundary $u$ of $J_g$ to create a new half-disk $J + \fm(g)$, as in Figure~\ref{fig:action}. 
    
    Define $J+\fm(-g)$ for $g\neq1$ similarly, but using the opposite sheet choice: push off $a_-$ and $b_+$ instead.
\end{defn}
For readers not familiar with 4-dimensional maneuvers, we give a precise meaning to the above constructions. 
They are best understood in terms of a stratification of the space of all maps $\Map(\HD, X;k)$ of half-disks. The open, dense (codimension~0) stratum are the \emph{generic immersions} whose singularities are finitely many interior double points that are transverse. 
A path in $\Map(\HD, X;k)$ can be perturbed to a finite concatenation of paths of generic immersions (that can be implemented by self-isotopies of domain and range) and the following three types of paths $J_s,s\in[0,1]$ (or their reverses) that meet the codimension~1 strata transversely in a single point.
\begin{itemize}
\item 
    A \emph{finger move} is a regular homotopy $J_s, s\in\I$, so that the individual $J_s\colon\HD\imra X$ are generic immersions, except for one time $s=s_0$ when $J_{s_0}$ has one singularity: an interior double point that is not transverse but with tangent spaces meeting in a 1-dimensional subspace. Thus, one sheet moves by an ambient isotopy along a path $g$ before it exhibits a self-tangency and right after creates a pair of additional transverse double points $y_\pm$ as in Figure~\ref{fig:fm}.
\item 
    \emph{Pushing off a double point $y_+$} across the free boundary is another finger move $J_s$ so that the individual $J_s$ are generic immersions, except for one time $s=s_0$ when $J_{s_0}$ has one singularity: an intersection $J_{s_0}(a)=J_{s_0}(b_{s_0})$ between the interior point $a\in\HD\sm\partial\HD$ and the free boundary point $b_{s_0}\in\partial\HD$. Thus, the sheet of $J_0$ around $a\in J_0^{-1}(y_+)$ moves along the dashed arc $b_s$ so that $J_s(a)=J_s(b_s)$ for $s\in [0,s_0]$, whereas this double point disappears for $s>s_0$, as in Figure~\ref{fig:push-across}.
\item
    A \emph{cusp homotopy} $J_s$ as in Figure~\ref{fig:interior-twist} creates an interior double point with trivial group element. There is a single time $s=s_0$ when $J_s$ has a singularity: at the cusp point the differential of $J_{s_0}$ has rank only 1. We say that $J_{1}$ is obtained from $J_0$ by an \emph{interior twist}.
\end{itemize}
Note that in Definition~\ref{def:half-action} the homotopy class of the half-disk remains unchanged, $J\simeq J\pm\fm(r)^G$, and that using \emph{distinct sheet choices} is essential: pushing off $a_+$ and $a_-$ instead (or $b_+$ and $b_-$), gives a half-disk isotopic to $J$ (use an additional parameter to push off the tangency in the finger move).
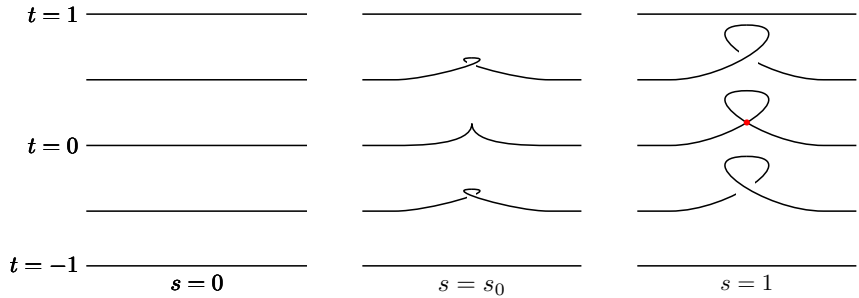
\begin{figure}[!htbp]
    \centering
    \begin{tikzpicture}[scale=1,every node/.style={scale=0.6},even odd rule]
    \clip (-1.7,-0.6) rectangle (1,2.1);
    \foreach \x in {2, 1.4, 0.8, 0.2, -0.3} {
        \draw[thin]
            (-1.01,\x) -- (1.01,\x);
        \draw[thin] 
            (-1,2) node[left]{\small$t=1$}
            (-1,0.8) node[left]{\small$t=0$}
            (-1,-0.3) node[left]{\small$t=-1$}
            (0,-0.3) node[below]{\small$s=0$};
    }
\end{tikzpicture}~\quad
\begin{tikzpicture}[scale=1,every node/.style={scale=0.65},even odd rule]
    \clip (-1,-0.6) rectangle (1,2.1);
    \draw[thin] (-1,2) -- (1,2); 
    \draw[thin] 
        (0,1.6) to[out=180,in=180,distance=0.3cm] 
        (0.7,1.4) -- (1,1.4); 
    \fill[white]  
        (-0.04,1.4) rectangle (0.04,1.57);
    \draw[thin] 
        (-1,1.4) -- (-0.7,1.4)
        to[out=0,in=0,distance=0.3cm] (0,1.6);
    \draw[thin] 
        (-1,0.8) -- (-0.7,0.8) to[out=0,in=-90,distance=0.2cm]
        (0,1) to[out=-90,in=180,distance=0.2cm] 
        (0.7,0.8) -- (1,0.8);
    \draw[thin] 
        (-1,0.2) -- (-0.7,0.2)
        to[out=0,in=0,distance=0.3cm] (0,0.4);
    \fill[white]  
        (-0.04,0.2) rectangle (0.04,0.37);
    \draw[thin] 
        (0,0.4) to[out=180,in=180,distance=0.3cm] 
        (0.7,0.2) -- (1,0.2);  
    \draw[thin] (-1,-0.3) -- (1,-0.3) node[below=2pt,pos=0.5]{\small$s=s_0$};
    
\end{tikzpicture}~\quad
\begin{tikzpicture}[scale=1,every node/.style={scale=0.6},even odd rule]
    \clip (-1,-0.6) rectangle (1.1,2.1);
    \draw[thin] (-1,2) -- (1,2); 
    \draw[thin] 
        (0,1.9) to[out=180,in=180,distance=0.55cm] 
        (0.7,1.4) -- (1,1.4); 
    \fill[white]  
        (-0.07,1.5) rectangle (0.1,1.8);
    \draw[thin] 
        (-1,1.4) -- (-0.7,1.4)
        to[out=0,in=0,distance=0.55cm] (0,1.9);
    \draw[thin] 
        (-1,0.8) -- (-0.7,0.8) to[out=0,in=0,distance=0.55cm]
        (0,1.3) to[out=180,in=180,distance=0.55cm] 
        (0.7,0.8) -- (1,0.8);
    \draw[thin] 
        (-1,0.2) -- (-0.7,0.2)
        to[out=0,in=0,distance=0.55cm] (0,0.7);
    \fill[white]  
        (-0.1,0.3) rectangle (0.07,0.6);
    \draw[thin] 
        (0,0.7) to[out=180,in=180,distance=0.55cm] 
        (0.7,0.2) -- (1,0.2);  
    \draw[thin] (-1,-0.3) -- (1,-0.3) node[below,pos=0.5]{\small$s=1$};
    \fill[red,draw]
        (0,1.01) circle (0.6pt);
\end{tikzpicture}
    \caption{Movies in time $t$ of the nonregular homotopy $J_s$ from a local disk to the interior twist on it. From right to left pull all arcs tight and observe a cusp at $s=s_0$, $t=0$.}
    \label{fig:interior-twist}
\end{figure}

The notation $J+\fm(r)$ hides a number of choices made in the construction, for example the choice of disjointly embedded arcs from the double points $y_\pm$ to the free boundary. It is possible to show geometrically that the isotopy class of the resulting half-disk is well defined and depends only on $J$ and $r$; however, we show this by the following indirect reasoning, which uses the observation that any choice comes with a canonical homotopy $J_s$ from $J$ to $J+\fm(r)$ (move the finger, then perform the push-offs). 
\begin{prop} \label{prop:Dax-inverts}
    For all $r\in \RGR$ and any choices in the construction of $J+\fm(r)$, the relative $\Dax$ invariant is given by
    \[
\Dax(J+\fm(r),J)=[r] \in \faktor{\RGR}{\md(\pi_3X)}
    \]
    As a consequence, the isotopy class of the half-disk $J+\fm(r)$ only depends on the isotopy class of $J$ and on $[r]$.
\end{prop}
\begin{proof}
    By definition, $\Dax(J+\fm(g),J)= \Dax(J_s)$ is the sum of double points of arcs that foliate a homotopy $J_s\colon\HD\to X$ from $J_0=J+\fm(r)$ to $J_1=J$.
    
    \begin{figure}[!htbp]
        \centering
        \begin{tikzpicture}[scale=0.65,every node/.style={scale=1.1},even odd rule]
    \clip (-1.5,-2.65) rectangle (13.45,5.8);
        \draw[line width=0.7pt, blue!80] 
             (-1,-2) to[out=70,in=180,distance=3cm] (6, 1.4)  to[out=0,in=110,distance=3.1cm] (13,-2);
        \draw[line width=0.7pt, green!60!black,postaction = decorate,decoration = {markings, 
        mark = at position 0.8 with {\arrow{latex}} }] 
            (-1,-2) to[out=70,in=180,distance=2.8cm] (8.3, 0.4)  to[out=0,in=110,distance=2cm] (13,-2);
        \fill[white] 
            (2.15,-0.2) rectangle (2.6,0.35)
            (3.6,0) rectangle (4,0.5)
            (2.4,0.8) rectangle (2.85,1.3)
            (3.6,1) rectangle (4,1.5)
            (7.1, 0.2) rectangle (7.4, 0.6)
            (7, 1.1) rectangle (7.4,1.6)
            (8.17, 1) rectangle (8.57,1.5);
        \draw[line width=0.6pt, postaction = decorate, decoration = {markings, 
        mark = at position 0.8 with {\arrow{latex}} }]
            (-1,-2) to[out=50, in=-160, distance=1.1cm] (2,-0.5)
            (4,0) -- (7.1,0) 
            (7.5,0) -- (8.8,0) to[out=0, in=125, distance=1.6cm] (13,-2);
        \draw[line width=0.7pt,postaction = decorate,decoration = {markings, 
        mark = at position 0.95 with {\arrowreversed{latex}} }]
                        (2,-0.5) to[out=20,in=-100, distance=0.6cm]
                        (3,3.15) to[out=80,in=98, distance=4cm] (8.3,2)
                        (4,0) to[out=180,in=-100, distance=0.5cm]
                        (4,3) to[out=80,in=85, distance=3cm] (7.25,1.8);
        \draw[line width=0.7pt]
                (7.25,1.8) to[out=-95,in=-110, distance=1.75cm] (8.15,-0.2)
                (8.25,0.15) to[out=80,in=-80, distance=0.6cm] (8.3,2);
        \fill[color=black!30!white] (5.6,3) circle (0.55cm);
        \draw[postaction = decorate,decoration = {markings, 
        mark = at position 0.55 with {\arrowreversed{latex}} }] 
            (5.6,3) circle (0.65cm) node[left=0.35cm]{$g$};
        \draw[line width=0.7pt, green!60!black, postaction = decorate, decoration = {markings, 
        mark = at position 0.8 with {\arrow{latex}} }] 
            (-1,-2) to[out=30,in=180,distance=3cm] (7.4, -0.4) 
            to[out=0,in=150,distance=2.1cm] (13,-2);

        \draw[line width=0.7pt, red!80] (-1,-2)  -- (13,-2);
        \draw 
            (-1,-2) circle (1.1pt) node[below]{$x_-$}  
            (13,-2) circle (1.1pt) node[below]{$x_+$};
        \draw[thick,dashed,orange] 
            (7.4, -0.4) -- (5.6, -2) 
            (-0.2,-2) to[out=50, in=-160, distance=0.6cm]
            (1.9,-0.4) to[out=20,in=-100, distance=0.6cm]
            (2.9,3.15) to[out=80,in=98, distance=4.2cm] (8.4,2)
            to[out=-80,in=70, distance=0.4cm] (8.3, 0.4);
        \fill[thick,orange] (-0.2,-2) circle (2pt) (5.6, -2) circle (2pt);
        \fill[thick]
            (7.4, -0.4) circle (2pt) (7.45, -0.95) node{$y_+$} 
            (8.3, 0.4) circle (2pt) (8.8, 0.65)  node{$y_-$};
        \draw[thick]
            (9.75,-2) node[below,red!80]{$u$} 
            (11.65,0.45) node[blue!80]{$k_-$}
            (9.95,-1.93) node[above=0.3cm,green!60!black]{\small$\alpha_+$} 
            (11,-2) node[above=0.55cm]{\small$\alpha_0$}
            (10.55,-2) node[above=1.16cm,green!60!black]{\small$\alpha_-$};
\end{tikzpicture}
        \caption{Arcs $\alpha_-,\alpha_0,\alpha_+$ in our foliation of a homotopy $J_s$ for the action, and dashed arcs guiding the push-offs.}
        \label{fig:fm-foliation}
    \end{figure}
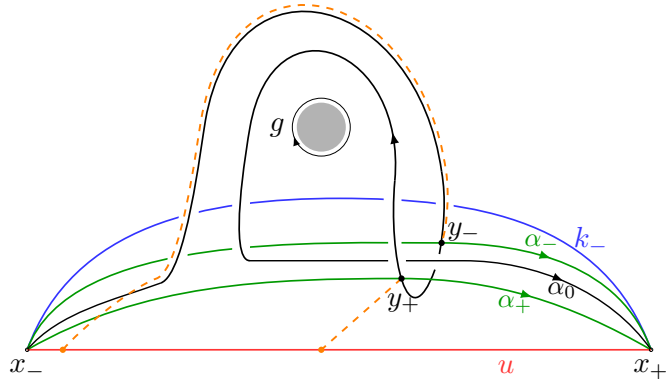
    The homotopy in the description of the finger move mentioned above is one such choice and we now foliate it conveniently, see Figure~\ref{fig:fm-foliation}.
    If $J_{s_0}$ denotes the moment of self-tangency, when the finger touches the disk, we foliate it so that a single arc, call it $\alpha_0$, has this point as a transverse self-intersection. Then $J_s$ for some nearby $s<s_0$ has two double points $y_-,y_+$ which must occur on two distinct arcs $\alpha_-,\alpha_+$ of the foliation, so that $\alpha_+$ is closer to $u$ than $\alpha_0$ and $\alpha_-$ is further from it. We pick the guiding arc for the pushing-across operation for $y_+$ to be the shortest path from $y_+$ to $u$, so that pushing never produces self-intersecting arcs. For the guiding arc for $y_-$ we pick the part of $\alpha_0$ from $y_-$ all the way close to $x_-$, and then a short arc from there to $u$. Similarly, all arcs stay embedded during this push.
    
    There is clearly only one immersed arc in this 2-parameter family, namely $\alpha_0$ in the half-disk $J_{s_0}$, and we claim that its unique double point has group element $+g$, so that $\Dax(J_s)=g$. Indeed, this arc looks precisely like the one in the realization map family $\realmap(g)$, see \cite[Fig.4.21]{KT-highd} and cf.\ Figure~\ref{fig:g+g-inv}.
    
    The final claim now follows from the exact sequence in Proposition~\ref{prop:action=r}: The isotopy class of a half-disk is determined by its homotopy class and $\Dax$ invariant, and have just shown $J\simeq J+\fm(r)$ and $\Dax(J,J+\fm(r))=[r]$.
\end{proof}


\begin{prop}\label{prop:action=r}
    For any half-disk $\U\in\HD(X;k)$, our action induces a map $\U+\fm(\bull):\RGR\to \HD(X;k)$ that translates under the bijection \eqref{eq:augm-vs-non} $\HD(X;k)\cong\pi_1(\Emb_\partial(\D^1,X),u)$ to the realization map of \cite{KT-highd}. As a consequence, we have a group extension:
\[\begin{tikzcd}
    \faktor{\RGR}{\md(\pi_3X)} \arrow[tail,shift left]{rr}{\U+\fm(\bull)} 
    && \HD(X;k) \arrow[dashed,shift left]{ll}{\Dax} \arrow[two heads]{rr}{-\U\cup\bull} 
    && \pi_2X
    \end{tikzcd}
\]
  and $\Dax$ inverts the action map on homotopic half-disks.
\end{prop}
\begin{proof}
    We have seen in Theorem~\ref{thm:Dax-arcs} that the realization map $\realmap$ is inverted by the invariant $\Da$ for arcs, whereas our geometric action is inverted by the invariant $\Dax$ for half-disks by Proposition~\ref{prop:Dax-inverts}. By Lemma~\ref{lem:Dax=Da} these invariants agree, so $\realmap$ translates to $\U+\fm(\bull)$.
    The exact sequence now follows from the one in Proposition~\ref{prop:central}.
\end{proof}
\begin{figure}[htbp!]
    \centering
    \begin{tikzpicture}[scale=1,every node/.style={scale=1},even odd rule]
    \clip (-1.1,-1.7) rectangle (1.15,2.7);
    \draw[thin] (-1,2.6) -- (1,2.6); 
    \draw[thin] 
        (0,2.5) to[out=180,in=180,distance=0.5cm] 
        (0.7,2) -- (1,2); 
    \fill[white]  
        (-0.1,2.1) rectangle (0.05,2.4);
    \draw[thin] 
        (-1,2) -- (-0.7,2)
        to[out=0,in=0,distance=0.5cm] (0,2.5);
    \draw[thin] 
        (0.08,1.5) to[out=140,in=0,distance=0.09cm]  (0,1.7) to[out=180,in=180,distance=0.5cm] 
        (0.7,1.2) -- (1,1.2);
    \fill[white]  
        (0.65,1.15) rectangle (0.75,1.25);       
    \draw[thin]
        (-1,1.2) -- (0.2,1.2) to[out=0,in=180,distance=0.2cm]
        (0.6,1.06) to[out=0,in=0,distance=0.15cm]  (0.6,1.3)
        (0.08,1.5) to[out=-40,in=180,distance=0.15cm] 
        (0.6,1.3);
    \draw[thin] 
        (0.08,0.7) to[out=140,in=0,distance=0.09cm]  (0,0.9) to[out=180,in=180,distance=0.5cm] 
        (0.75,0.4) -- (1,0.4);
    \draw[thin]
        (-1,0.4) -- (0.2,0.4) to[out=0,in=180,distance=0.2cm]
        (1,0.26) to[out=0,in=0,distance=0.15cm]  (1,0.5)
        (0.08,0.7) to[out=-40,in=180,distance=0.3cm] 
        (1,0.5);
    \draw[thin]
        (-1,-0.4) -- (0.2,-0.4) to[out=0,in=180,distance=0.2cm]
        (0.6,-0.54) to[out=0,in=0,distance=0.15cm]  (0.6,-0.3)
        (0.08,-0.1) to[out=-40,in=180,distance=0.15cm] 
        (0.6,-0.3);
    \fill[white]  
        (0.65,-0.45) rectangle (0.75,-0.35);
    \draw[thin] 
        (0.08,-0.1) to[out=140,in=0,distance=0.09cm]  (0,0.1) to[out=180,in=180,distance=0.5cm] 
        (0.7,-0.4) -- (1,-0.4);
    \draw[thin] 
        (-1,-1.2) -- (-0.7,-1.2)
        to[out=0,in=0,distance=0.5cm] (0,-0.7);
    \fill[white]  
        (-0.1,-1.1) rectangle (0.05,-0.8);
    \draw[thin] 
        (0,-0.7) to[out=180,in=180,distance=0.5cm] 
        (0.7,-1.2) -- (1,-1.2);  
    \draw[thin] (-1,-1.6) -- (1,-1.6);
    
    \foreach \x in {2.6, 2, 1.2, 0.4, -0.4, -1.2, -1.6} {
        \draw[thin]
            (-1.01,\x) circle (0.55pt)
            (1.01,\x) circle (0.55pt);
    }
\end{tikzpicture}
    \caption{Movie of the half-disk $\U_{tw}$.}
    \label{fig:boundary-twist}
\end{figure}
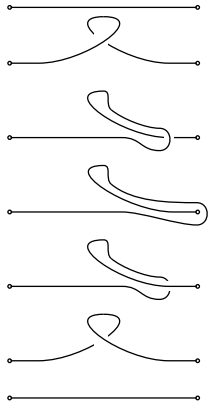
\begin{cor}
\label{cor:U_tw}
    Let $\U_{tw}\colon\HD\hra X$ be the  half-disk obtained from $\U$ by one interior twist, followed by pushing the resulting double point off the free boundary $u$, as in Figure~\ref{fig:boundary-twist}. Then $\U_{tw}$ and $\U$ are isotopic half-disks in $X$.
\end{cor}
\begin{proof}
    These disks are homotopic by construction: first use a cusp homotopy on $\U$ and then a finger move across the free boundary. Moreover, their relative $\Dax$ invariant vanishes because we are not counting the trivial group element. Thus, Proposition~\ref{prop:action=r} implies the statement.
\end{proof}
\begin{remark}
    One can \emph{see} an isotopy from $\U_{tw}$ to $\U$ as follows: Rotate the arcs in Figure~\ref{fig:boundary-twist} around their right endpoints (comprising $u$) by $360$ degrees in the plane of the paper and isotope the over/under-strands accordingly. Then pull each strand tight.
\end{remark}

\subsection{From half-disks to neat disks}\label{sec:half-to-neat}

We have a commutative diagram
\begin{equation}\label{eq:augm-vs-non}
\begin{tikzcd}[column sep=small]
    \Dk \arrow{r}{\cong} \arrow[two heads]{d}{pr} & \HD(M_G;k^\e) \arrow{r}{\cong} \arrow[two heads]{d}{} 
    & \pi_1(\Emb^\e(\D^1,M_G;k_0),u^\e) \arrow[two heads]{d}{\ev^\e}\\
    \faktor{\Dk}{\langle \U_{tw}^G \rangle} \arrow{r}{}
    & \HD(M_G;k) \arrow{r}{\cong} 
    & \pi_1(\Emb(\D^1,M_G;k_0),u)
\end{tikzcd}
\end{equation}
where the top row is \eqref{eq:bijections} (implied by Theorem~\ref{thm:space-LBT}), and uses augmented boundary conditions for half-disks and augmented arcs in $M_G$. In the bottom row we have the simpler set of half-disks $\HD(M_G;k)$ with the non-augmented boundary condition, which is in turn homotopy equivalent to the space of non-augmented arcs in $M_G$. The bottom right arrow is another foliation map from \cite{KT-highd}, for the non-augmented case.

The bottom left arrow exists since $G$ bounds a 3-ball in $M_G$, so under the map $\Dk\cong\HD(M_G;k^\e)\sra \HD(M_G;k)$ the neat disk $\U_{tw}^G$ turns into the half-disk $\U_{tw}$, which is isotopic to $\U$ in $\HD(M_G;k)$ by Lemma~\ref{cor:U_tw}. We next show this map is an isomorphism.
    
Recall that \eqref{eq:eta-W-pr} says that the kernel of the rightmost vertical arrow is $\Z$, and that this sequence splits by a homomorphism $\eta_W$. We now define an analogous splitting for disks. For the tangent bundle of the universal cover of $M$ we have the Stiefel-Whitney class $w_2(\wt M) \in H^2(\wt M;\Z/2)^\pi\cong \Hom_\pi(\pi_2M,\Z/2)$. Since $G$ has trivial normal bundle in $M$, this induces a homomorphism
\[\begin{tikzcd}
    \Dk\rar{-\U\cup\bull} & \faktor{\pi_2M}{\langle G \rangle} \rar{w_2} & \Z/2,
\end{tikzcd}
    \quad  K\mapsto w_2(\wt M)(-\U\cup K).
\]
which is the mod 2 reduction of the Euler homomorphism $e_\U$ from~\eqref{eq:e-U}. Thus, the division by 2 in the following definition makes sense.
\begin{defn}
    Define a map $\eta_{W,\U}\colon\Dk\to\Z$ by
    \[
        \eta_{W,\U}(K)=\frac{1}{2}(e_\U(K) - W(-\U\cup K)),
    \]
    where $W\in H^2(\wt M;\Z)^\pi \cong \Hom_\pi(\pi_2M,\Z)$ is a $\pi$-equivariant integer lift of $w_2(\wt M)$ such that $W(G)=0$. 
\end{defn}

\begin{prop}
\label{prop:splitting}
    The bijection $\Dk\cong\HD(M_G;k^\e)$ from \eqref{eq:bijections} reduces to the bijection $\Dk/\langle \U_{tw}^G \rangle \cong \HD(M_G;k)$. Moreover
    \[
        \eta_{W,\U} \times pr\colon \Dk \overset{\cong}{\ra}\Z \times \faktor{\Dk}{\langle \U_{tw}^G \rangle}
    \]
    is an isomorphism of groups.
\end{prop}
\begin{proof}
    The top left bijection in \eqref{eq:augm-vs-non} induces a bijection on the fibers of the vertical surjections,
    so it suffices to show that $\eta_{W,\U}$ is a splitting.
    
    Thus, we just need to check that $\eta_{W,\U}$ takes $\U_{tw}^G$ to 1. Firstly, $e_\U(\U_{tw}^G)=2$ because: interior twist changes the framing by 2, finger move does not change it, neither does tubing into $G$ as it is framed. Secondly, $W(-\U\cup \U_{tw}^G)=0$ since $\U_{tw}^G$ is homotopic (rel.\ boundary) to $\U\#(-G)$ by Lemma~\ref{lem:homotopy-class}, and $W(G)=0$. Together we get the desired $\eta_{W,\U}(\U_{tw}^G)=1$.
\end{proof}

A lift $W$ as in Proposition~\ref{prop:splitting} exists because any oriented 4-manifold has a spin$^c$ structure~\cite{Vogt-T}. Alternatively, in our setting such $W$ was constructed in~\cite[Prop.B.18]{KT-highd}.
Using Proposition~\ref{prop:splitting} and a small diagram chase we get the following consequence for the group $\Dk^0\coloneqq\ker(e_\U)$ from \eqref{eq:Dk0} that is relevant in the study of mapping class groups in Section~\ref{sec:mcg-proofs}.
\begin{cor}\label{cor:w2}
    There is an exact sequence of groups
    \[
        \{1\} \ra \Dk^0 \ra \faktor{\Dk}{ \langle \U_{tw}^G \rangle} \overset{w_2}{\ra} \Z/2
    \]
    that is short exact if and only if $\wt M$ is not spin.
\end{cor}
Recall that we extend the involution $\sigma(g)=\ol{g}$ for $g\neq 1$ in a nonstandard way by setting $\sigma(1)\coloneqq0$. This makes the next result true: for $r=1$ it says that $K+\fm(1)^G=K_{tw}^G \simeq K \# (-G)$ for the $+1$-interior twist in Figure~\ref{fig:interior-twist}.
\begin{lemma}\label{lem:homotopy-class}
    For a neat disk $K$ in $M$ the geometric action of $r\in \Z[\pi]$ changes its homotopy class by a connected sum with $(\sigma(r)-r)$ copies of $G$, that is, $K+\fm(r)^G$ is homotopic to $K \# (\sigma(r)-r)\cdot G$ rel.\ boundary.
\end{lemma}
\begin{proof}
    Under the above correspondence with half-disks, during the homotopy from $K$ to $K+\fm(g)$ we cross $u$ twice, which adds two copies of $G$. For the positive double point $y_+$ we use the $a$-sheet so the group element guiding the tube into $G$ is clearly $g$, but the sign is $-1$. Indeed, $\mathrm{sign}_{y_+}(\D^2_a,\D^2_b)=+1$ and also $\lambda(K,G)=1$, so we can see in the 3-dimensional model of Figure~\ref{fig:push-across} that we have to tube into the negatively oriented $G$. The argument is analogous for the negative double point $y_-$, giving $+\ol{g}$.
\end{proof}

\subsection{Main proofs}\label{sec:main-proofs}
In this section we prove the remaining results from the introduction, \ref{thm-intro:Dax-values}, \ref{prop-intro:4-term}, \ref{thm-intro:groups} and \ref{prop:commutators}. Recall that Theorem~\ref{thm-intro:Dax-classifies} is a clear consequence of Theorem~\ref{thm-intro:Dax-values}, and Lemma~\ref{lem:Dax-invt} was proven as Lemma~\ref{lem:HDax} and Corollary~\ref{cor:Dax-FQ}, whereas the proof of Lemma~\ref{lem:M-c} was given in Section~\ref{sec:dax-proofs}. Therefore, we can now concentrate on disks in the setting $(M,\U,G)$.
\begin{proof}[Proof of Theorem~\ref{thm-intro:Dax-values}]
    The bijection between half-disks in $M_G$ and disks in $M$ takes the operation of ``pushing across the free boundary'' into ``tubing into the dual $G$'', so the $\RGR$-action on half-disks from Definition~\ref{def:half-action} translates to the $\RGR$-action on disks from the introduction (see Figure~\ref{fig:action}). Their $\Dax$ invariants clearly agree (cf.\ the last paragraph of Section~\ref{sec:Dax-half}). Therefore, putting $X=M_G$ in Proposition~\ref{prop:Dax-inverts} implies the analogous statement for neat disks: the action $+\fm(\bull)^G$ of the group $\RGR$ on $\Dk$ satisfies
    \[
        \Dax(J+\fm(r)^G,J)=r\in\faktor{\RGR}{\md(\pi_3M)}.
    \]
    Thus, given homotopic neat disks $K_0,K_1\in\Emb(\D^2,M;k)$  whose relative $\Dax$ invariant $\Dax(K_0,K_1)=\RGR^\sigma/\md(\pi_3M)$ is trivial, the extension in Proposition~\ref{prop:central} with $\U=K_1$ implies that they are are isotopic. 
    Finally, Lemma~\ref{lem:homotopy-class} says that $r\in \RGR$ preserves the homotopy class if and only if $r\in \RGR^\sigma$, i.e.\ is fixed under the involution. 
\end{proof}

Before proving Propositions~\ref{prop-intro:4-term}, \ref{thm-intro:groups} and \ref{prop:commutators}, we state and prove Proposition \ref{prop:group-diagram}. However, let us first recall from Definition~\ref{def:action} that the notation 
\[\begin{tikzcd}[cramped,column sep=17pt] 
        G\arrow[squiggly]{r}{a} & S\arrow[two heads]{r} & S/G 
    \end{tikzcd}
\]
means that we have a group action $a\colon G \times S \to S$ of a group $G$ on a set $S$, with orbit set $S/G$. 
Note that the arrow labeled by $a$ is not a map, but a shortcut for the group action;
once we choose $s\in S$ we do have a map $G\to S$, $g\mapsto a(g,s)$. 
One way to get a short exact sequence of a group action is when $S$ is a group and the action comes from a group homomorphism $\rho\colon G\to S$ via $a(g,s)\coloneqq\rho(g)\cdot s$, and $H=\ker(\rho)$. This exhibits $S$ as a group extension of $G/H$ by $S/G$ if $\im(\rho)\trianglelefteq S$ is a normal subgroup. 
 
For example, if $p\colon E\to B$ is a fibration with a basepoint $e\in E$ the lifting property gives a $\pi_1(B,b)$-action on $\pi_0(F)$, where $b\coloneqq p(e)$ and $F\coloneqq p^{-1}(b)$ is the fiber over $b$. The stabilizer of a point $x\in F$ is the image of $\pi_1(E,x)$ in $\pi_1(B,b)$. In our notation this is a short exact sequence of the group action
\[\begin{tikzcd}
    \pi_1(B,b) \arrow[squiggly]{r}{a} & \pi_0(F)\arrow[two heads]{r} & \faktor{\pi_0(F)}{\pi_1(B,b)} \cong \pi_0(p)^{-1}[b] \subseteq \pi_0(E).
\end{tikzcd}
\]
This action is usually written as an exact sequence of groups/sets, where one uses $e$ to get from the group action to a map $\pi_1(B,b) \to \pi_0(F), g\mapsto g(e)$:
\[
    \dots \to \pi_1(E,e) \to \pi_1(B,b) \to \pi_0(F) \to \pi_0(E) \to \pi_0(B)
\]
The fibration that gives our group actions on disks in a 4-manifold arises from the inclusion $\Emb\subseteq \Map$ of embeddings of arcs into all maps (rel.\ boundary) as in diagram \eqref{eq:Map-seq}. If $H$ is the homotopy fiber of that inclusion then we get a fibration sequence $\Omega \Emb \to \Omega \Map \to H$ where the loop spaces are taken at a neat embedding $u\colon\D^1 \hra X$ as usual. Applying the general case of a fibration above to this case gives the group action
\[\begin{tikzcd}
        \pi_1(H,u) \arrow[squiggly]{r}{a} & \pi_0(\Omega \Emb)\arrow[two heads]{r} & \faktor{\pi_0(\Omega\Emb)}{\pi_1(H,u)} \cong \pi_0(\Omega\Map) 
\end{tikzcd}
\]
that can be rewritten via $\pi_1(H,u)\cong \pi_2(\Map,\Emb,u)\cong\RGR$ as 
\[\begin{tikzcd}
        \RGR \arrow[squiggly]{r}{a} & \pi_1(\Emb,u)\arrow[two heads]{r} & \faktor{\pi_1(\Emb,u)}{\RGR} \cong \pi_1(\Map,u)\cong \pi_2X 
\end{tikzcd}
\]
This action leads to Proposition~\ref{prop:central} by identifying the stabilizers of the action, given by the image of $\pi_1(\Omega \Map)\cong \pi_2(\Map)\cong\pi_3X$, with $\md_u(\pi_3X)$.

\begin{prop}\label{prop:group-diagram}
    In the setting $(M,\U,G)$ there is a commutative diagram of short exact sequences of group actions, with the connecting map from the upper right to the lower left equal to the identity:
    \[\begin{tikzcd}[column sep=3em]
        &&  \faktor{\Z[\pi]}{\RGR^\sigma} \arrow[tail,squiggly]{d}{\# (\sigma-\Id)\cdot G} 
    \\
        \faktor{\RGR^\sigma}{\md(\pi_3M)} \arrow[tail,squiggly]{r}{+\fm^G} \arrow[tail,squiggly]{d} &
        \Dk \arrow[two heads]{r}{j}\arrow[equal]{d} &
       \mu_2^{-1}(0)\arrow[two heads]{d}{p_\U}  
    \\
        \faktor{\Z[\pi]}{\md(\pi_3M)} \arrow[two heads]{d} \arrow[tail,squiggly]{r}{+\fm^G} &
        \Dk \arrow[two heads]{r}{p_\U} & 
        \faktor{\pi_2M}{\Z[\pi]\cdot G} 
    \\
        \faktor{\Z[\pi]}{\RGR^\sigma} & & 
    \end{tikzcd}
    \]
    Here $p_\U(K) = [-\U\cup K]\pmod{\Z[\pi]\cdot G}$, as in Theorem~\ref{thm-intro:groups}.
\end{prop}
\begin{proof}
    The exactness of the lower horizontal sequence is immediate from Proposition~\ref{prop:action=r} and the bijection $\Dk\cong\HD(M_G;k)$. Indeed, we identified the $\fm$-actions and Dax invariants for neat and half-disks in the proof of Theorem~\ref{thm-intro:Dax-values}, and we have $\md(\pi_3M_G)=\md(\pi_3M)$ by Lemma~\ref{lem:Dax-M-M_G} and $\pi_2M_G\cong\pi_2M/\Z[\pi]\cdot G$ by Lemma~\ref{lem:lambda-splits}.

    To see that the right vertical sequences is exact, we extend it to the right:
\[\begin{tikzcd}
        & \faktor{\Z[\pi]}{\RGR^\sigma} \arrow[tail]{d}{\U \# (\sigma(\bull)-\bull)\cdot G} \arrow[tail]{r}{\sigma-\Id} 
        &  \Z[\pi] \arrow[tail]{d}{\U \# (\bull)\cdot G} \arrow[two heads]{r} 
        &  \faktor{\Z[\pi]}{\im(\sigma-\Id)}  \arrow[equal]{d} 
    \\
        \Dk \arrow[two heads]{r}{j}
        & \mu_2^{-1}(0) \arrow[two heads]{d}{p_\U} \arrow[tail]{r} 
        & {[\D^2,M;k]} \arrow[two heads]{d}{p_\U} \arrow[two heads]{r}{\mu_2} 
        &  \faktor{\Z[\pi]}{\langle 1, \ol{g}-g \rangle} 
    \\
        & \faktor{\pi_2M}{\Z[\pi]\cdot G} \arrow{r}{\cong}
        &   \pi_2M_G 
        &
    \end{tikzcd}
\]
    and observe that the new vertical sequence is short exact by Lemma~\ref{lem:lambda-splits}, and the two new horizontal 3-term sequences are short exact by inspection.

    The desired connecting map is computed as follows. Start with $g\in\pi_1M$ representing a generator of the group $\Z[\pi]/\Z[\pi]^\sigma$ in the lower left corner. Acting by $g$ on $\U$ gives $\U+\fm(g)^G\in  \Dk$, which is by $j$ mapped to the homotopy class $[\U+\fm(g)^G] =[\U \# (\sigma(g)-g)\cdot G]$ by Lemma~\ref{lem:homotopy-class}. By definition, this is also the image of $g$ under the map from the upper right corner of the diagram. Thus, its connecting map is the identity.
    
    The upper horizontal sequence is now exact by a diagram chase.
\end{proof}

\begin{proof}[Proof of Proposition~\ref{prop-intro:4-term}]
    The exactness of the group action follows from Theorem~\ref{thm-intro:Dax-values}, except that we still need to prove that $\mu_2$ is onto and $\im(j)=\mu_2^{-1}(0)$. The latter follows from the Norman trick using the existence of our dual sphere $G$. To prove the former, turn a null homotopy of $k$ into a generic neat immersion $J\colon\D^2\imra M$. Then the boundary condition implies that $\lambda(J,G) = 1$ and hence for any $r\in\RGR$ we get
    \[
\mu_2(J \# r\cdot G) = \mu_2(J) + \lambda(J, r\cdot G) = \mu(J) + r.
    \]
    This shows that $\mu_2$ is surjective and in particular, the value $0$ is attained. Using the Norman trick, this also shows that $\Dk$ is not empty.
\end{proof}

\begin{proof}[Proof of Theorem~\ref{thm-intro:groups}]
    The extension in the statement is precisely the lower exact sequence in Proposition~\ref{prop:group-diagram} (which was obtained from the extension of Proposition~\ref{prop:action=r}). To identify the inverse of the action $+\fm^g$ on the kernel of $-\U\cup\bull$ we use the splitting result in Proposition~\ref{prop:splitting} and the inverse $\Dax$ from Proposition~\ref{prop:action=r}. The only subtlety is that we only know that $K$ and $\U$ are homotopic modulo $G$, so we cannot use the Dax invariant for neat disks as in the rest of the introduction. However, we can use the Dax invariant for half-disks and then the result follows from those two propositions.
\end{proof}
 
    
\begin{proof}[Proof of Proposition~\ref{prop:commutators}]
    In Proposition~\ref{prop:central} we computed the commutator pairing for the group $\pi_1\Emb_\partial(\D^1,M_G)\cong\HD(M_G,k)$, but not the $\e$-augmented version $\HD(M_G;k^\e)\cong\Dk$. However, we also showed that the latter group is a product with $\Z$, see Proposition~\ref{prop:splitting}. Since the generator of this $\Z$ is given by $\U_{tw}^G\in\Dk$ which is central, it follows that the commutator pairing for $\Dk$ is the same.
\end{proof}

\begin{remark}\label{rem:future-gp-str}
    In upcoming work, we will actually show that the sequence in Proposition~\ref{prop-intro:4-term} can also be equipped with group structures after choosing an undisk $\U$. 
    Given the quadratic property~\ref{eq:mu-2-quadratic} of Wall's self-intersection invariant, it is surprising to find a group structure on $[\D^2,M;k]$ for which $\mu_2$ becomes a homomorphism. Via $K\mapsto (-\U\cup K)$, this will give the \emph{twisted group structure} $a_1\star a_2= a_1 + a_2 - \lambda(a_1,a_2)\cdot [G]$ on $\pi_2M$. However, $\mu_2$ does \emph{not} become a homomorphism on $\pi_2M$, it instead stays quadratic! 
    Nevertheless, the map $a\mapsto \mu_2(a) + \lambda(a,\U)$ is a homomorphism $(\pi_2M,\star) \to \Z[\pi]/\langle 1, \ol{g}-g \rangle$, since $\mu_2(-\U\cup K) = \mu_2(K) - \lambda(\U,K)$. 

    We compare the group $(\pi_2M,\star)$ to Proposition~\ref{prop:commutators}, which computed the commutator $[K_1,K_2]=\U +\fm(r)^G$ with $r=\lambdabar(a_1, a_2)\in\RGR$, for $K_i\in\Dk$ and $a_i=[-\U\cup K_i]\in\pi_2M$. Since
    $\U +\fm(r)^G\simeq \U \# (\ol{r}-r)G$, it follows that $-\U\cup[K_1,K_2]\simeq -\U\cup(\U \# (\ol{r}-r)\cdot G)\simeq (\ol{r}-r)\cdot G$.
    This agrees with the description of $\star$ as $\lambda$ is hermitian and the coefficients of $1$ cancel.
\end{remark}

\subsection{An extension of mapping class groups}\label{sec:mcg-proofs}
In this section we prove Theorem~\ref{thm:mcg}, a further result from Section~\ref{sec:further results}. In the setting $(M,\U,G)$ let $\U'$ denote the half-disk in $M_G\coloneqq M \cup_G h^3$ corresponding to $\U$. Extending diffeomorphisms by the identity over the additional handle gives inclusions $i_G\colon\Diff_\partial(M)\hra\Diff_\partial(M_G)$ and
$i_\U\colon\Diff_\partial(M\sm\nu\U)\hra\Diff_\partial(M)$.
\begin{lemma}\label{lem:i_Gi_U}
   The composite map $i_G\circ i_\U$ is a weak homotopy equivalence. Therefore, $\pi_ki_G\circ\pi_ki_\U$ are isomorphisms for all $k\geq0$ and in particular, $\pi_ki_\U$ is split injective and $\pi_ki_G$ is split surjective.
\end{lemma}
\begin{proof}
    The inclusion $M\sm\nu\U\hra M\hra M_G$ is an embedding isotopic to a diffeomorphism, since $h^2=\nu\U$ and $h^3$ form a canceling pair of handles by~\eqref{eq:MG}. Namely, $Y'\coloneqq\nu\U\cup h^3=M_G\sm(M\sm\nu\U)$ is diffeomorphic to a half-ball, whose boundary is the union of $3$-disks $D$ in $\partial(M\sm\nu\U)$ and $Z'$ in $\partial M_G$. Thus, we can use $Y'$ to isotope $D$ to $Z'$, ending up with the claimed diffeomorphism.
    
    Now, $i_G\circ i_\U\colon \Diff_\partial(M\sm\nu\U)\hra\Diff_\partial(M_G)$ is a weak homotopy equivalence by Cerf's Theorem~\cite[Prop.2.10]{KT-highd} for $Y=X=M_G$, $y=\Id_{M_G}$, $Z=\partial M_G$. Indeed, $Y'\subseteq M_G$ is a local normal tube to $\partial M_G$ along $Z'$, so $\Emb_{Z}(Y,X;y)=\Diff_\partial(M_G)$, while $\Emb_{Z\cup Y'}(Y,X;y)$ is precisely the image of $i_G\circ i_\U$.
\end{proof}

In the following commutative diagram, all maps in the square on the right, as well as the vertical map at the bottom left, are restriction maps of embedding spaces to various submanifolds. Therefore, they are fibrations by Cerf's Theorem, see \cite[Thm.2.9]{KT-highd}.
\[
\begin{tikzcd}
    \Diff_\partial(M\sm\nu\U)\arrow{d}[swap]{i_\U}\arrow[dashed]{dr}{\simeq}
    \\
         \Diff_\partial(M) \arrow{d}[swap]{\text{act on } \U} \arrow{r}{i_G} 
        &  \Diff_\partial(M_G) \arrow{d}{\text{act on } \U'} \arrow[]{r}{\text{act on } h^1} 
        &  \Emb_{\partial \D^1 \times \D^3}(\D^1 \times \D^3, M_G)  \arrow{d}{\ev_{\D^1 \times [0,\e]}}
    \\
        \Emb(\D^2, M;k) \arrow{r} 
        & \Emb(\HD, M_G;k_-) \arrow[]{r}{\ev_{\D^\e_+}} 
        &  \Emb_\partial^\e(\D^1, M_G)
    \end{tikzcd}
\]
Here $h^1$ is the 1-handle in $M_G$ that is dual to the 3-handle $h^3$ that was attached to $M$. In other words, $h^1$ runs from $\partial M_G$ to itself and removing it brings us back to $M$.
The surjectivity of $\pi_ki_G$, $k\geq0$, from Lemma~\ref{lem:i_Gi_U}, implies that $\pi_k(\text{act on } h^1)$ are trivial, and the connecting maps are injective (and are given by ambient isotopy extension). Similarly, $\pi_ki_\U$ is injective (and a section for $\pi_ki_G$), so the map $\pi_k(\text{act on }\U)$ is surjective, except for $k=0$, where we need to determine the image. For this we extend the diagram down as follows:
\[\begin{tikzcd}[column sep=0.8cm]
        \pi_1 \Emb_{\partial\D^1 \times \D^3}(\D^1 \times \D^3, M_G) \arrow[tail]{d}{\pi_1\ev_{\D^1 \times [0,\e]}} \arrow[tail]{r}{\pi_0\amb}
        & \pi_0 \Diff_\partial(M) \arrow{d}{\pi_0(\text{act on } \U)} \arrow[two heads]{r}{\pi_0i_G} 
        & \pi_0 \Diff_\partial(M_G) 
    \\
        \pi_1 \Emb_\partial^\e(\D^1, M_G) \arrow{r}{\pi_0\amb_\U}[swap]{\cong} \arrow{d}{\delta}
        & \pi_0 \Emb(\D^2, M;k) \arrow{d}{e_\U} &
    \\
       \pi_0 \Omega \S^1 \arrow{r}{\cong}
       & \Z &
    \end{tikzcd}
\]
Here $\delta$ is the connecting map for the fibration $\ev_{\D^1\times[0,\varepsilon]}$, whose fiber is homotopy equivalent to $\Omega \S^1$; since the normal bundle of $u^\e$ is 2-dimensional, its unit sphere is $\S^1$). Therefore, as $\pi_1\Omega\S^1=0$, the map $\pi_1\ev_{\D^1\times[0,\e]}$ is injective.

To see that the bottom square in the last diagram commutes, note that $\delta$ is by definition given by ambiently extending a loop of $\e$-augmented arcs to a loop of ``thickened'' arcs, which amounts to completing the given loop in the Stiefel bundle $V_2(M_G)$ to a $V_3(M_G)$. In $\pi_0\Omega\S^1$ this counts the framing given by that additional vector. If the loop is ambiently extended from $\U$ to a disk $K\in\Emb(\D^2,M;k)$, this 
additional vector is tangent to $K$ and this number becomes the relative framing of the normal bundle of $K$ with respect to $\U$ (along their boundary).

Hence the domain of $\pi_0\amb$ maps isomorphically onto $\ker(\delta)\cong\ker(e_\U)$ which we denoted by $\Dk^0$ in \eqref{eq:Dk0}. We thus get the sequence claimed in Theorem~\ref{thm:mcg}:
\[\begin{tikzcd}
       \Dk^0 \arrow[tail]{r}{a_{\U}}
        & \pi_0 \Diff_\partial(M)  \arrow{r}{\pi_0i_G} \arrow[bend right]{l}[swap]{\pi_0(\text{act on }\U)}
        & \pi_0 \Diff_\partial(M_G) \arrow[bend right]{l}[swap]{\pi_0i_\U}
    \end{tikzcd}
\]
Note that $a_{\U}\coloneqq\pi_0\amb\circ(\pi_1\ev_{\D^1\times[0,\varepsilon]})^{-1}\circ\pi_0\foliate^\e_\U$ is a homomorphism since lifting a composition of loops carefully shows that $\amb$ is a homomorphism, and on $\Dk^0<\Dk\coloneqq\pi_0 \Emb(\D^2, M;k)$ the group structure by construction comes from that of $\pi_1 \Emb_\partial^\e(\D^1, M_G)$ via $\pi_0\amb_\U$.
However, acting on $\U$ gives only a point-theoretic splitting $\pi_0(\text{act on }\U)$ of $a_{\U}$. \qed

\newpage
\phantomsection
\section*{Tables of notation}\setcurrentname{Tables of notation}\label{sec:notation}

\addcontentsline{toc}{section}{Tables of notation}

\begin{center}
\small
\begin{tabular}{ c|c|c|c|c|c }
    \parbox[c][0.8cm][c]{1.9cm}{\centering Object} & \parbox[c][1.3cm][c]{2.4cm}{\centering Notation for smooth maps/ embeddings} &  
    \parbox[c][0.8cm][t]{1.7cm}{\centering Boundary condition} & 
    \parbox[c][1.3cm][t]{1.5cm}{\centering Isotopy classes of emb.\ } & 
    \parbox[c][1.3cm][t]{1.5cm}{\centering Homotopy classes of maps}
    & Find in
    \\\hline
    neat disks in $X$ 
        & \parbox[l][1cm][c]{2.4cm}{\centering $\Map(\D^2,X;k)$ $\Emb(\D^2,X;k)$} & $k\colon\S^1\hra\partial X$ & $\Dk$ & $[\D^2,X;k]$
        & \parbox[c][1cm][c]{1.9cm}{\centering 
        Thm.~\ref{thm-intro:Dax-values} Thm.~\ref{thm:space-LBT}} \\\hline
    \parbox[c][1cm][c]{1.9cm}{\centering spheres in $N$ dual to $G$} 
        & \parbox[l][1cm][c]{2.4cm}{\centering $\Map(\S^2,N)^G$ $\Emb(\S^2,N)^G$} & / & $\SG$ & $[\S^2,N]^G$
        & 
        Sec.~\ref{sec:spheres} \\\hline
    half-disks in $X$
        & \parbox[l][1.1cm][c]{2.4cm}{\centering 
        $\Map(\HD,X;k)$ $\Emb(\HD,X;k)$} & \parbox[l][1.6cm][c]{2.2cm}{\centering
        $k= k_-\cup u$
        $k_-\colon\D^1\hra\partial X$ $u\colon\D^1\overset{neat}{\hra} X$} & $\HD(X;k)$ & $[\HD,X;k]$
        & \parbox[c][1cm][c]{1.9cm}{\centering 
        Thm.~\ref{thm:space-LBT}
        Def.~\ref{def:boundary conditions}}\\\hline
    neat arcs in $X$ &        
        \parbox[l][1cm][c]{2.4cm}{\centering 
        $\Map_\partial(\D^1,X)$
        $\Emb_\partial(\D^1,X)$} & $k_0\colon\S^0\hra\partial X$ & / & / 
        & Sec.~\ref{sec:arcs}\\ 
    \hline
\end{tabular}
\end{center}

\medskip

\renewcommand{\arraystretch}{1.5}
\begin{center}
\small
\rotatebox{0}{
\begin{tabular}{c|c|c}
\multicolumn{3}{c}{Dax invariants for arcs}\\\hline
    \multirow{2}{*}{$\Da$} 
    & $\pi_2\big(\Imm_\partial(\D^1,X),\Emb_\partial(\D^1,X),u\big)\ra\Z[\pi_1X\sm1]$ 
    & Def.~\ref{def:Dax}\\
    & $ \ker\big(\pi_1\big(\Emb_\partial(\D^1,X),u\big)\sra\pi_2X\big)\ra\Z[\pi_1X\sm1]/\md_u(\pi_3X)$
    & Eq.~\ref{eq:Da-final} \\ \hline
\multicolumn{3}{c}{Dax homomorphisms}\\\hline
    $\md_u$
    & $\md_u\coloneqq\Da\circ\delta_{\Imm}\colon\;\pi_3X \ra \Z[\pi_1X\sm1]$  
    & Def.~\ref{def:dax}
    \\
    $\md$
    & $\md\coloneqq\md_u\colon\;\pi_3X \ra \Z[\pi_1X\sm1]^\sigma$ if $u$ is homotopic into $\partial X$.
    & Eq.~\ref{eq:dax-def}
    \\\hline
\multicolumn{3}{c}{Dax invariants for homotopic half-disks}\\\hline
    \multirow{2}{*}{$\Dax$}
   & \{homotopies between half-disks in $\Emb(\HD,X;k)\}\ra\Z[\pi_1X\sm1]$
    & Def.~\ref{def:Dax(H)}  \\
    & \{two homotopic half-disks from $\HD(X;k)\}
    \ra\Z[\pi_1X\sm1]/\md(\pi_3X)$
    & Lem.~\ref{lem:HDax}\\\hline
\multicolumn{3}{c}{Dax invariants for homotopic neat disks}\\\hline
    \multirow{2}{*}{$\Dax$}
    &  \{homotopies between neat disks in $\Emb(\D^2,X;k)\}
    \ra\Z[\pi_1X\sm1]^\sigma$
    & \multirow{2}{*}{Rem.~\ref{rem:Dax-neat}}  \\
    & \{two homotopic neat disks from $\D(X;k)\}
    \ra\Z[\pi_1X\sm1]^\sigma/\md(\pi_3X)$
    & \\\hline
\multicolumn{3}{c}{Freedman--Quinn invariant for homotopic spheres and neat disks}\\\hline
    \multirow{2}{*}{$\FQ$}
    & \{two homotopic spheres from $\S(N)^G\}
    \ra\FF_2[T_N]/\mu_3(\pi_3N)$  
    & Prop.~\ref{prop:disks-spheres}
    \\
    & \{two homotopic neat disks from $\D(X;k)\} \ra\FF_2[T_X]/\mu_3(\pi_3X)$  
    & Def.~\ref{def:FQ}
    \\\hline
\multicolumn{3}{c}{Wall intersection and reduced self-intersection invariants}\\\hline
$\lambda$
    & hermitian intersection form $\lambda\colon\pi_2X\times\pi_2X \ra\Z[\pi_1M]$  
    & \multirow{2}{*}{Eq.~\ref{eq:lambdabar}}\\
    $\lambdabar$
    & reduced intersection form $\lambdabar\colon\pi_2X\times\pi_2X \ra\Z[\pi_1M\sm1]$  \\
    $\mu_2$
    & quadratic refinement $\mu_2\colon[\D^2,M;k] \ra\Z[\pi_1M\sm1]/\langle\ol{r}{-}r\rangle$  
    & Prop.~\ref{prop-intro:4-term}
    \\
    $\mu_3$
    & 3-dimensional analogue $\mu_3\colon[\D^3,P^6] \ra\Z[\pi_1P\sm1]/\langle\ol{r}{+}r\rangle$  
    & Def.~\ref{def:mu_3}
    \\\hline
\end{tabular}
}
\end{center}

\newpage
\phantomsection
\printbibliography[heading=bibintoc]

\vspace{10pt}

\hrule

\end{document}